\documentclass[11pt]{amsart}
  
\usepackage{mathrsfs}  
\usepackage[backref,colorlinks]{hyperref} 
\usepackage{amsmath,amssymb,amsthm,bm}
\usepackage[foot]{amsaddr}
\usepackage[abbrev,msc-links]{amsrefs}   
\usepackage[utf8]{inputenc}
\usepackage{tikz}
\usepackage{enumerate}
\usepackage{enumitem}
\usepackage{graphicx}
\usepackage{epstopdf}
\usepackage{lineno}

\usepackage{mathtools}
\DeclarePairedDelimiter\ceil{\lceil}{\rceil}
\DeclarePairedDelimiter\floor{\lfloor}{\rfloor}

\DeclarePairedDelimiter{\final}{\langle}{\rangle}

\DeclareMathOperator{\dist}{dist}
\DeclareMathOperator{\diam}{diam}
\DeclareMathOperator{\height}{ht}

\theoremstyle{plain}
\newtheorem{thm}{Theorem}[section]
\newtheorem{fact}[thm]{Fact}
\newtheorem{prop}[thm]{Proposition}
\newtheorem{clm}[thm]{Claim}

\newtheorem{lem}[thm]{Lemma}

\newtheorem*{clm*}{Claim}

\theoremstyle{definition}
\newtheorem{rem}[thm]{Remark}
\newtheorem{dfn}[thm]{Definition}
\newtheorem{exmp}[thm]{Example}
\newtheorem{quest}[thm]{Question}
\newtheorem{conj}[thm]{Conjecture}
\newtheorem{obs}[thm]{Observation}

\numberwithin{equation}{section}

\normalsize                              
\let\polishlcross=\l
\def\l{\ifmmode\ell\else\polishlcross\fi}

\def\NN{\mathbb N}
\def\ZZ{\mathbb Z}

\def\cH{{\mathcal H}}

\def\cP{{\mathcal P}}

\def\tw{{\mathrm{tw}}}

\def\eps{\varepsilon}
\def\phi{\varphi}

\DeclareMathAlphabet\bfc{OMS}{cmsy}{b}{n}
\DeclareMathAlphabet{\pzc}{OT1}{pzc}{m}{it}

\title[Slow graph bootstrap percolation  II]{Slow graph bootstrap percolation  II:  \\Accelerating  properties}
\author{David Fabian $^{1,\ast}$}
\author{Patrick Morris $^{2,\dagger}$}
\author{Tibor Szab\'o $^{3,\ddagger}$}

\address{$^1$ Institut für Quantenphysik, Universität Ulm, Germany}
\address{$^2$ Departament de Matem\`atiques, Universitat Polit\`ecnica de Catalunya (UPC), Barcelona, Spain.}
\address{$^3$ Institute of Mathematics, Freie Universit\"at Berlin, Germany}

\thanks{$^\ast$ Research supported by the Deutsche Forschungsgemeinschaft (DFG)
Graduiertenkolleg “Facets of Complexity” (GRK 2434). }

\thanks{$^\dagger$ Research supported by   the DFG Walter Benjamin program - project number 504502205 and by the
Marie Sk{\l}odowska-Curie grant RAND-COMB-DESIGN (number 101106032). Views and opinions expressed are
however those of the author(s) only.}

\thanks{$^\ddagger$ Research
supported by the DFG under Germany’s Excellence Strategy - The Berlin Mathematics Research Center
MATH+ (EXC-2046/1, project ID: 390685689).}

\email{david.fabian@uni-ulm.de, pmorrismaths@gmail.com,  szabo@math.fu-berlin.de}

\date{\today}

\begin{document}

\begin{abstract}
For a graph $H$ and an $n$-vertex graph $G$, the $H$-bootstrap process on $G$ is the process which starts with $G$ and, at every time step, adds any missing edges on the vertices of $G$ that complete a copy of $H$. This process eventually stabilises and we are interested in the extremal question raised by Bollob\'as of determining the maximum \emph{running time} (number of time steps before stabilising) of this process over all possible choices of $n$-vertex graph $G$. In this paper, we initiate a systematic study of the asymptotics of this parameter, denoted $M_H(n)$, and its dependence on properties of the graph $H$.  Our focus is on $H$ which define relatively fast bootstrap processes, that is, with $M_H(n)$ being at most linear in $n$. We  study the graph class of trees, showing that one can bound $M_T(n)$ by a quadratic function in $v(T)$ for all trees $T$ and all $n$. We then go on to explore the relationship between the running time of the $H$-process and the minimum vertex degree and connectivity of $H$. 
\end{abstract}
\maketitle
\section{Introduction}
Given graphs $H$ and $G$, let $n_H(G)$ denote the number of copies of $H$ in $G$. The $H$\textit{-bootstrap percolation process}
($H$\emph{-process} for short)
on a graph $G$ is the sequence $(G_i)_{i\geq 0}$ of graphs defined by $G_0 := G$ and
    \begin{linenomath}
        \begin{equation*}
        V(G_i) := V(G), \qquad \quad
        E(G_i) := E(G_{i-1}) \cup \left\{e\in\binom {V(G)} 2 : n_H\left(G_{i-1}+e\right)>n_H(G_{i-1})\right\},
        \end{equation*}
    \end{linenomath}
for $i\geq 1$.
We call $G$ the \emph{starting graph} of the process and $\tau_H(G):=\min\{t\in \NN: G_t=G_{t+1}\}$ the \emph{running time} of the $H$-process on $G$, which is the point at which the process stabilises. Finally, we define $G_\tau$ with $\tau=\tau_H(G)$ to be the \emph{final graph} of the process and denote it $\final{G}_H:=G_\tau$.  

The $H$-bootstrap percolation process was introduced in 1968 by Bollob\'as~\cite{bollobas1968weakly} in his study of \emph{weak saturation} and has since been studied from many different viewpoints with connections being made to the general study of cellular automata and bootstrap processes~\cite{morris2017bootstrap}. In particular, inspired by similar questions from statistical physics~\cite{chalupa1979bootstrap}, Balogh, Bollob\'as and Morris~\cite{balogh2012graph} initiated the study of the case when the starting graph $G$ is the binomial random graph $G(n,p)$. In  both this probabilistic setting and the original extremal graph theory question of weak saturation the focus has been to determine under what conditions the final graph of the process is the complete graph. From a cellular automata perspective, where one views the process as a virus spreading, this translates to asking whether the virus will infect the  entire population.

More recently, there has been interest in asking \emph{how long} a virus will spread for, that is, how long does it take for a  percolation process to stabilise. This was studied by Gunderson, Koch and Przykucki~\cite{gunderson2017time} in the context of a random starting graph and Bollob\'as raised the extremal question of determining the maximum running time of the $H$-bootstrap process over all $n$-vertex starting graphs $G$. 

    \begin{dfn} \label{def:maxruntime}
    For $n\in\NN$, we define $M_H(n)$ to be 
	\[M_H(n):=  \max_{|{V(G)}|=n}\tau_H(G).\]
    \end{dfn}

We also mention that similar questions on running times have been recently studied for neighbourhood percolation, a cellular automata closely related to graph bootstrap percolation, both in the extremal setting~\cites{benevides2015maximum,przykucki2012maximal} and in the probabilistic setting of a random starting graph \cites{balister2016time,bollobas2014time,bollobas2015time}.

\subsection {Previous work}\label{sec:previous}
The initial study of the maximum running time of $H$-bootstrap processes focused on the case when $H$ is a clique. Indeed, Bollob\'as, Przykucki, Riordan and Sahasrabudhe~\cite{bollobas2017maximum} and, independently, Matzke~\cite{matzke2015saturation} showed that $M_{K_4}(n)=n-3$, for all $n\ge 3$. Moreover the first set of authors~\cite{bollobas2017maximum} gave constructions showing that $M_{K_r}(n)\ge n^{2-\lambda_r-o(1)}$ for $r\ge 5$, where $\lambda_r$ is some explicit constant such that $\lambda_r\rightarrow 0$ as $r\rightarrow \infty$.  However the same authors believed that there was a limit to how long the   $K_r$-bootstrap process could last and conjectured that for all $r\ge 5$, $M_{K_r}(n)=o(n^2)$. Balogh, Kronenberg, Pokrovskiy and the third author of the current paper~\cite{balogh2019maximum} then  disproved this conjecture, showing  that $M_{K_r}(n)=\Omega(n^2)$ for all $r\ge 6$. For $r=5$, using an interesting connection to Behrend's construction of $3$-term arithmetic progression free sets, they could show that $M_{K_5}(n)$ grows greater than $n^{2-\eps}$ for all $\eps>0$ but it is unknown whether the rate of growth is quadratic or not. In the first paper of our series \cite{FMSz1} (see also \cite{fabian2022maximum}) we determined $M_H(n)$ exactly for all cycles $H=C_{k}$ with $ 3\leq k \in \NN$ (the $k=3$ case was known and is an exercise), showing that $M_{C_k}(n)$ is of the order $\log_{k-1}(n)$ and that the exact form of the function depends on the parity of $k\in \NN$. Finally, in recent papers, Noel and Ranganathan \cite{noel_running_2022}, Hartarsky and Lichev \cite{hartarsky_maximal_2022}, and Espuny Díaz, Janzer, Kronenberg and Lada \cite{espuny_diaz_long_2022} extended the study of $M_H(n)$ to hypergraphs, again focusing on the case where $H$ is a clique and providing lower bounds.

\subsection{A preview} \label{sec:preview}
Building on  \cites{balogh2019maximum,bollobas2017maximum}, in forthcoming work~\cite{FMSz3} we develop a general framework for providing lower bounds on $M_H(n)$   using so-called \emph{chain constructions}. 
We use our framework to give  very effective lower bounds for many different $H$,  often asymptotically meeting the trivial upper bound of $n^2$. 
Indeed using our framework we will show that, in a very strong sense, \emph{almost all} graphs $H$ have the property that $M_H(n)=\Theta(n^2)$. In more detail we show that for $k\in \NN$ and $H=H(k,p)$ the binomial random graph such that $p=\omega(\log k/k)$, with probability tending to 1 as $k$ tends to infinity, $H$ has the property that $M_H(n)=\Theta(n^2)$ (asymptotics for $M_H$ are with respect to $n$ here and throughout). Similarly, we  show quadratic running time for $H$ that have a large minimum degree $\delta(H) > 3v(H)/4$. Even for certain graphs $H$ where we can show that the running time is subquadratic (we will show this for all bipartite graphs for example), our chain constructions will show that $M_H(n)$ grows superlinearly for many natural $H$ such as $H=K_{k,\ell}$ with $k,\ell\geq 3$. 

These results suggest that in order to have a  maximum running time that is at most linear, $H$ must have \emph{atypical} and \emph{sparse} structural properties. The purpose of this paper is to explore such properties and the extent to which they accelerate the running times of $H$-processes. For simplicity in what follows   we restrict our attention to \emph{connected} $H$ but discuss briefly the disconnected case in our concluding remarks, see Section \ref{sec:disconnect}.  Our starting point is a natural class of sparse atypical graphs, namely \emph{trees}.

\subsection{Trees} \label{sec:intro-trees}
Our first main result shows that for any tree $T$, the $T$-process stabilises after  constantly many steps, that is, the running time is  independent of the number of vertices $n$ of the starting graph. Our   upper bound gives a quadratic dependence on the number $t$ of vertices of $T$.

    \begin{thm}\label{thm:trees}
    Let $t\in\NN$ and $n\geq 2t$. Every tree $T$ on $t$ vertices satisfies
        \[
        M_T(n) \leq \frac 1 8 \cdot (t^2 + 6t + 68).
        \]
    \end{thm}
 
For certain simple examples of a tree $T$, such as a path or a star, it does not take long to convince oneself that the $T$-process stabilises in constant time. Furthermore, given that a neighbour of a leaf in a copy of the target tree $T$ in $G_1$ becomes an almost universal vertex 
in just one round of the percolation process and hence the diameter of $G_2$ already becomes constant, the statement of Theorem~\ref{thm:trees} perhaps does not come across as unexpected. The actual proof however, that from constant diameter the percolation process  finishes in constant time, turned out to be a considerable challenge. Indeed,  the analysis needed to handle all the great variety of trees and choices of starting graph is quite delicate. Moreover, extra ideas were needed to achieve an upper bound whose dependence on the number of vertices $t$ of the tree is polynomial.

\vspace{2.4mm}

What affects the maximum running time of $H$-percolation? 
Theorem \ref{thm:trees} shows that $M_T(n)$ is constant for trees $T$ and by \cite{FMSz1} it is logarithmic for cycles. These are in contrast with cliques, where percolation can be almost quadratically slow already for $K_5$ (and linear for $K_4$). 
Considering that trees $T$ have \emph{minimum degree} $\delta(T)=1$, \emph{vertex connectivity} $\kappa(T)=1$ and  \emph{edge connectivity} $\kappa'(T)=1$ while $\kappa(C_k)=\kappa'(C_k)=\delta(C_k)=2$ for any $k\geq 3$, one cannot help but wonder whether these parameters  being at most two has anything to do with sublinear running time. 
In this work we initiate the study of target graphs with relatively fast percolation and the relationship between running times and small minimum degree/connectivity.  We obtain several necessary and sufficient conditions, refute a couple of natural conjectures, as well as arrive at a number of tantalizing open problems.

\subsection{Small degree vertices} \label{sec:small degrees}
In our analysis of the running time for trees, the existence of leaves plays a crucial role. Moreover, the following simple example shows that adding a vertex of degree one can hugely accelerate the running time. 

    \begin{exmp} \label{ex:clique pendent edge}
    Let $k\geq 3$ and $H_k$ be the $k+1$-vertex graph formed by taking a clique of size $k$ and adding a pendent edge to one of its vertices. Then $M_{H_k}(n)\leq 3$.
    \end{exmp}
    \begin{proof}
    Take $G$ to be an arbitrary $n$-vertex starting graph with $H_k$-process $(G_i)_{i\geq 0}$ and suppose that $\tau_H(G)\geq 1$ and so there is some copy of $H_k$ in $G_1$. Let $U\subset V(G)$ be the set of $k$ vertices that form a clique in this copy of $H_k$. Then in $G_2$, we have that $U$ and $V(G)\setminus U$ form parts of a complete bipartite graph. Thus any pair of vertices in $V(G)\setminus U$ appear as an edge in $G_3$ and so $G_3$ is a complete graph. 
    \end{proof}

This suggests a connection between fast running times and vertices of small degree,  in particular vertices of degree one. Our next main result shows that such small degree vertices are indeed \emph{necessary} for a sublinear running time. 

    \begin{thm}\label{thm:23bound}
    Let $H$ be a connected graph with minimum degree $\delta(H)\geq 2$ and maximum degree $\Delta(H)\geq 3$.
    Then
        \[
        M_H(n) = \Omega(n).
        \]
    Moreover, if $H$ is bipartite then there exists a bipartite starting graph $G$ with $\tau_H(G)=\Omega (n)$. 
    \end{thm}
    
Theorem \ref{thm:23bound} shows that the only connected graphs $H$ which have $\delta(H)\geq 2$ and sublinear $M_H(n)$ are cycles, which are known to have $M_H(n)$ logarithmic~\cite{FMSz1}. The next proposition shows that Theorem \ref{thm:23bound} is tight in that the lower bound on running time cannot be improved. Indeed    the complete bipartite graph $K_{2,s}$ has $\delta(K_{2,s})=2$ and maximum degree arbitrarily large and yet  only linear maximum running time.  
    
    \begin{prop}\label{prop:k2s_upper}
    For every $s\geq 3$, $M_{K_{2,s}}(n) = \Theta(n)$.
    \end{prop}  

\subsection{Constructions with small degrees}
Theorem \ref{thm:23bound} shows that apart from cycles, a condition of $\delta(H)=1$  is \emph{necessary} for a sublinear running time. Given Example \ref{ex:clique pendent edge}, which shows that a single pendent edge can reduce the running time from quadratic to constant, one might also expect that $\delta(H)=1$ is   a \emph{sufficient} condition. Somewhat surprisingly, this turns out not to be the case. Indeed, there is a graph with a pendent edge whose running time is in fact quadratic.

    \begin{thm} \label{thm:simulate}
    For each $1\leq t\in \NN$, there is a connected graph $H$ with $\delta(H)=t$ and $M_H(n)=\Omega(n^2)$. 
    \end{thm}

Note that Theorem \ref{thm:simulate} is only really of interest for $t\leq 4$; otherwise we already have examples of $H$ with quadratic running times by considering cliques. 
%

So in our quest to characterise connected graphs $H$ with sublinear running time, what extra condition, in addition to the minimum degree being $1$, {\em guarantees} sublinear running time? 
The graph in Theorem~\ref{thm:simulate} has a vertex of full degree (of degree $15$). The construction in our next result shows that even limiting a largest degree to $3$ is of no help.

    \begin{thm}\label{thm:counterexample}
    There exists a connected graph $H$ with minimum degree $\delta(H)=1$ and maximum degree $\Delta(H)=3$ satisfying $M_H(n) = \Omega(n)$.
    \end{thm}

We remark that the maximum degree condition cannot be improved to 2 as paths have constant running time (Theorem \ref{thm:trees}). The proof of Theorem~\ref{thm:counterexample} is a highly more intricate manifestation of the {\em simulation} construction method we develop to establish Theorem~\ref{thm:simulate}.  Although we state Theorem \ref{thm:counterexample} as a single instance of a graph $H$, our construction can easily be adapted to give an infinite family of such graphs $H$.

\subsection{Low connectivity}\label{sec:connect}
As graphs $H$ with $\delta(H)=t$ also have vertex connectivity $\kappa(H)\leq t$ and edge connectivity $\kappa'(H)\leq t$, by Theorem \ref{thm:simulate}, bounds on these parameters are not sufficient to give upper bounds on the maximum running time of the $H$-process. We can however recover an effective bound using low connectivity if we have the added condition that the $H$-process on $H$ itself has a final graph $\final{H}_H$ which is complete. This is the content of our final theorem presented here. 

    \begin{thm}\label{thm:girth_construction}
    If $H$ is a nonempty graph such that there exists $e\in E(H)$ with $\kappa(H-e)\leq 2$ and $\final{H}_H=K_{v(H)}$, then
        \[
        M_H(n) = O(n).
        \]
    \end{thm}

It turns out that such a connectivity condition is in fact also \emph{necessary} in order to have a running time that is at most linear. Indeed, using our framework for chain constructions, in forthcoming work~\cite{FMSz3} we show that if $H$ is such that for any edge $e\in E(H)$, $\kappa(H-e)\geq 3$ (in particular if $\kappa(H)\geq 4$), then there exists $\eps_H>0$ such that $M_H(n)=\Omega(n^{1+\eps_H})$. This  extends Theorem \ref{thm:23bound} which can be seen as saying that in order for a connected graph $H$ which is not a cycle to have sublinear maximum running time, it needs to have vertex (and edge) connectivity 1. 

Although the condition that $\final{H}_H=K_{v(H)}$ in Theorem \ref{thm:girth_construction} is quite restrictive, the theorem can still be used to generate interesting examples. One such example is $K_5^-$, that is, $K_5$ minus a single edge, which is in contrast to $K_5$ which has (almost) quadratic running time \cite{balogh2019maximum}. We give further examples as follows.  

    \begin{figure}[h]
    \centering
    \includegraphics[width=0.4\linewidth]{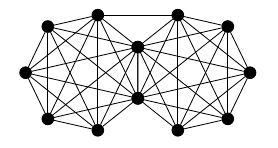}
    \caption{An illustration of the graph $H'_7$.  }
    \label{fig:connectivity}
    \end{figure}

    \begin{exmp} \label{ex:H'k}
    Let $k\geq 3$ and  $H'_k$ be the $(2k-2)$-vertex graph composed by `gluing together' two cliques of size $k$ along a singular edge $e$ and adding one more edge $e'$ between two non-adjacent vertices. (An illustration of $H'_7$ is given in Figure \ref{fig:connectivity}).    Then $M_{H'_k}(n)=\Theta(n)$. 
    \end{exmp}

    \begin{proof}
    The $H'_k$-process on $H'_k$ results in $K_{2k-2}$ after just one step of the process as any missing edge can play the role of $e'$ in a copy of $H'_k$. Note, moreover that $\kappa(H'_k-e')=2$ as removing the vertices of the edge $e$ disconnects $H'_k-e'$. Therefore $H'_k$ satisfies the assumptions of Theorem \ref{thm:girth_construction} and so  indeed $M_{H'_k}(n)=\Theta(n)$ where we used Theorem \ref{thm:23bound} for the lower bound on running time here. 
    \end{proof}

Note that $\delta(H'_k)=k-1$ and so we can generate graphs $H$ with arbitrarily large minimum degree $\delta(H)$ and linear $M_H(n)$. Therefore in order to generalise Theorem \ref{thm:23bound} to show that certain conditions are necessary for a running time which is \emph{at most linear}, one should focus on low connectivity conditions (as discussed above and shown in \cite{FMSz3}) rather than small degree conditions.

\subsection{Summary} \label{sec:summary}
Before embarking on our proofs, we briefly summarise our findings. Our theorem (Theorem \ref{thm:trees}) shows that trees $T$ have $M_T(n)$ bounded by a constant depending on $T$ and independent of $n$, and this constant can be taken to be quadratic in $v(T)$. The existence of small degree vertices and hence low connectivity in $H$  is \emph{necessary} for a sublinear running time (Theorem \ref{thm:23bound}). Small degree vertices stop being necessary when we consider $H$ with linear $M_H(n)$ (Example \ref{ex:H'k}) but low connectivity is still necessary to have running time at most linear (as shown in \cite{FMSz3}). Finally, neither small vertex degrees nor low connectivity are \emph{sufficient} to guarantee even a subquadratic running time (Theorem \ref{thm:simulate}) but with an extra condition concerning the $H$-process on itself, a bounded connectivity condition can recover a bound on running time (Theorem \ref{thm:girth_construction}).

\subsection*{Organisation of the paper.}
Necessary notation is introduced in Section \ref{sec:notation}.
The proof of  Theorem \ref{thm:trees} is  then given in Section \ref{sec:trees}.  Theorem \ref{thm:23bound} is proven in Section \ref{sec:min2max3} followed by a proof of Proposition \ref{prop:k2s_upper} in Section \ref{sec:k2s_upper}. Theorems \ref{thm:simulate} and \ref{thm:counterexample} are shown in Section \ref{sec:counterexample}. 
Finally Theorem \ref{thm:girth_construction} is proven in Section \ref{sec:girth} and in Section \ref{sec:conclude} we discuss various future directions of research.

\section{Notation and Preliminaries}\label{sec:notation}
Here we introduce some necessary terminology and notation and give a simple observation which will be useful throughout our proofs.  Firstly, for integers $0\leq i\leq j\in \NN$, we let $[i,j]=\{i,i+1,\ldots,j\}$ denote the discrete interval between $i$ and $j$.

\subsection*{Graph notation} For an edge $e\in E(H)$ we define $H-e$ as the graph obtained by removing $e$ from the edge set and similarly, if $e$ is an edge (of $K_{v(H)}$) on the same vertex set as $H$ then $H+e$ is the graph obtained by adding $e$ to $H$.

We denote the (external) \emph{disjoint union} of two graphs $G$ and $G'$, by $G \sqcup G'$, that is,
    \[
    G \sqcup G' := \left( (V\times\{1\}) \cup (V'\times\{2\}) \;,\; \{(x,1)(y,1) : xy \in E(G)\} \cup \{(x,2)(y,2) : xy \in E(G')\} \right).
    \]
If $G'\subseteq G$ is a subgraph of a graph $G$ and $X,Y$ are (not necessarily disjoint) subsets of\footnote{Note that we do \emph{not} require that $X$ and $Y$ are subsets of $V(G')$ here.} $V(G)$ we write
    \[
    E_{G'}(X,Y) = \left\lbrace  xy \in E(G') : x\in X, y \in Y \right\rbrace. 
    \]
We write $N_{G'}(v)$ for the set of neighbours of $v$ in $G'\subseteq G$ and $\deg_{G'}(v)=|N_{G'}(v)|$ for the degree of $v$ in $G'$. When  $G'\neq  G$ we sometimes write $G'$-neighbour or $G$-neighbour to emphasize the set of neighbours we are referring to.
We say that $v$ is \emph{universal} in $G$ if $N_G(v) = V(G)\setminus\{v\}$.

\subsection*{Graph bootstrap processes} 
Whenever the process $(G_i)_{i\geq 0}$ is clear from context, we say that a property of a graph holds \emph{at time} $i$ if $G_i$ has that property. We now give a simple observation that shows that processes behave well with respect to subgraphs. It can be proven easily by induction on $i\geq 0$.

    \begin{obs}\label{obs:hom}
    Let $ G' \subseteq G$  and let $(G'_i)_{i\geq0}$ and $(G_i)_{i\geq0}$ be the respective $H$-processes on $G'$ and $G$.
    Then $G'_i\subseteq G_i$ for every $i\geq 0$.
    \end{obs}

We say that a graph $G$ is $H$\emph{-stable} if $n_H(G+e) = n_H(G)$ for every $e\in \binom{V(G)}2\setminus E(G)$.
For any graph $G$ we define $\langle G \rangle_H$ to be the \emph{final graph} of the $H$-process on $G$.  By Observation \ref{obs:hom}, any $H$-stable graph containing $G$ must also contain every graph of the $H$-process on $G$. This implies the following observation. 

    \begin{obs} \label{obs:final}
    If $G$ is an $n$-vertex graph, $\langle G \rangle_H$ is the smallest (in terms of number of edges) $n$-vertex $H$-stable graph in which $G$ appears as a subgraph.
    \end{obs}

\subsection*{Rooted trees}
Our proof of Theorem \ref{thm:trees} relies on the concept of a \emph{rooted tree}.
A rooted tree is a tree $T$ together with a designated vertex $z \in V(T)$. For two vertices $u,v\in V(T)$, we denote by $\dist_T(u,v)$ the number of edges in the unique path from $u$ to $v$ in $T$. 
For $u \in V(T)$ the neighbours of $u$ whose distance to $z$ is larger than $\dist_T(z,u)$ are called \emph{children} of $u$.
If $u\neq z$ the unique neighbour $v\in N_T(u)$ with $\dist_T(z,v) < \dist_T(z,u)$ is called the \emph{parent} of $u$.
The \emph{height} of a vertex $u$ of $T$ with root $z$, denoted $\height_z(u)$, is the length of a longest  path $u_0\ldots u_\ell$ with $u=u_0$ such that $u_i$ is a child of $u_{i-1}$ for $i\in [\ell]$.
The height of $T$, denoted $\height_z(T)$, is defined as the height of the root $z$.

\subsection*{Tree embeddings}
Given a copy $T_0\subseteq G$ of a tree (or, more generally, a forest) $T$ in a graph $G$ and vertices $x\in V(T_0)$, $y \in V(G)\setminus V(T_0)$ we define the graph $T_0^{(x\to y)}$ via 
    \[
    V(T_0^{(x\to y)}) = V(T_0)\setminus \{ x \} \cup \{ y \},
    \qquad
    E(T_0^{(x\to y)}) = E(T_0) \setminus \{ e \in E(T_0) : x \in e \} \cup \{ yz : z \in N_{T_0}(x) \}.
    \]
We can think of $T_0^{(x\to y)}$ as the graph obtained by replacing $x$ with $y$ in $T_0$.
Note that $T_0^{(x\to y)}$ is a subgraph of $G$ if and only if $N_{T_0}(x) \subseteq N_G(y)$.
If $\varphi: V(T) \to V(G)$ is an embedding of $T$ into $G$ with $\varphi(T) = T_0$ we write $\varphi^{(x\to y)}$ for the map
    \[
    V(T) \to V(T_0^{(x\to y)})
    \;,\; \qquad 
    v \mapsto \begin{cases}
        \varphi(v) &,\; \phi(v)\neq x; \\
        y &,\; \phi(v)=x.
    \end{cases}
    \]

\section{Trees }\label{sec:trees}
In this section, we will prove  Theorem \ref{thm:trees}. 
We first have a look at the case when $T$ is a star.

    \begin{lem} \label{lem:star}
    For $T=K_{1,t-1}$, we have  that $M_T(n)\leq t-1$. 
    \end{lem}
    \begin{proof}
    In the $K_{1,t-1}$-process $(G_i)_{i\geq 0}$ with some arbitrary initial graph $G$, we say a vertex is \emph{central} at time $i$ if it plays the role of the centre of a copy of a star that lies in $G_i$ but misses precisely one edge in $G_{i-1}$.
    Note that if a vertex is central at time $i$ then it is universal at time $j$ for all $j\ge i$ and hence cannot be central at some time $k>i$. Hence at every step $i$, if $G_i\neq G_{i-1}$, then there is some vertex $v$ that is central (and hence universal) at time $i$ that was not universal at time $i-1$.  Therefore, if $\tau_{K_{1,t-1}}(G)\ge t-2$, then $G_{t-2}$ contains at least $t-2$ universal vertices and $G_{t-1}$ must be a complete graph. 
    \end{proof}

Now moving to general trees $T$,  we show the existence of a vertex cover with certain properties, as given by the following lemma. Recall that a vertex cover is a set of vertices such that each edge has  at least one endpoint in that set. 

    \begin{lem}\label{lem:trees_vertexcover}
	For a tree $T$ and a root vertex $z\in V(T)$ that is not a leaf, there exists a subset $U=U(T;z)\subset V(T)$ such that
    	\begin{enumerate}[label=(\arabic*)]
		\item $U$ is a smallest vertex cover of $T$;
		\item No vertex in $U$ is a leaf of $T$; 
		\item A vertex $v\in V(T)$ lies in $U$ if and only if it has a child that is not contained in $U$. 
		\end{enumerate}
    \end{lem}
    
    \begin{proof}
    Pick a smallest vertex cover $U$ that minimises $\sum_{u\in U}\dist_T(u,z)$ among all smallest vertex covers of $T$.
    Property (1) is satisfied by the choice of $U$.
    Property (2) is a direct consequence of (3) since leaves do not have children.
    The if direction of Property (3) holds because $U$ is a vertex cover.
    The only if direction of (3) will be achieved by contradiction:
    Suppose that $v\in U$ and all its children lie in $U$.
    Note that this includes the case when $v$ is a leaf.
    If $v$ is the root we can remove it from $U$ to arrive at a smaller vertex cover and hence at a contradiction.
    Suppose that $v\neq z$, and let $w$ be the parent vertex of $v$. If $w\in U$ then 
     $U\setminus\{ v \}$ is still a vertex cover and is smaller than $U$, which is a contradiction. Hence $w\notin U$ and $U\setminus\{ v \}\cup \{w\}$ is another smallest vertex cover.
    However,
    \begin{linenomath}     \begin{align*}
    \sum_{u\in U\setminus\{ v \}\cup\{ w \}} \dist_T(u,z) 
        = \sum_{u\in U}\dist_T(u,z) \;-\; \dist_T(v,z) + \dist_T(w,z) = \sum_{u\in U}\dist_T(u,z) \;-\; 1,
    \end{align*} \end{linenomath}
    which contradicts minimality of $U$.
    Therefore Property (3) must hold.
    \end{proof}

Next we define certain parameters which depend on our tree $T$ and our choice of root $z$ which is not a leaf of $T$. Let $U=U(T;z)$ be the vertex set given by Lemma \ref{lem:trees_almostuniversal}.  
We define
    \begin{equation}\label{eq:trees_mu}
    \mu=\mu (T;z) := M_{T[U]}(|U|).
    \end{equation}
That is, $\mu$ is the longest time the $T[U]$-process takes to stabilise when starting with an initial graph on $|U|$ vertices, where $T[U]$ is the forest induced by $T$ on the vertex cover $U$. 
    
Note that since $U$ is a smallest vertex cover of $T$, removing all but one of the vertices in $U$ from $T$ results in the disjoint union of a star and some isolated vertices. 
Let
    \begin{equation} \label{eq:trees_delta}
    \delta=\delta(T;z) := \min_{u\in U} |N_{T}(u)\setminus U|
    \end{equation}
be the smallest number of neighbours outside $U$ a vertex in $U$ can have.

We now give our main proposition, from which the proof of Theorem \ref{thm:trees} will follow. The proposition shows that if the $T$-process on some graph $G$ does not stabilise after some bounded number of steps (in terms of $\mu,\delta$ and $\height_z(T)$), then there will be many universal vertices. 

    \begin{prop} \label{prop:tree_many_univ}
    Let $T$ be a $t$-vertex tree with a root $z$ that is not a leaf and let $\mu=\mu(T;z)$, $\delta=\delta(T;z)$ and $U=U(T;z)$ be as defined in \eqref{eq:trees_mu}, \eqref{eq:trees_delta} and Lemma \ref{lem:trees_vertexcover} above. Moreover, let $(G_i)_{i\geq 0}$ be some $T$-process with $v(G_0)=n\geq 2t$ vertices.  Then if the process has not stabilised within the first 
        \begin{equation}\label{eq:trees_i*}
        i^* := 2+3\ceil{\height_z(T)/2}+\mu+\delta
        \end{equation}
    steps, $G_{i^*}$ contains $|U|+\delta-2$ universal vertices.
    \end{prop}

We now show how Theorem \ref{thm:trees} follows from Proposition \ref{prop:tree_many_univ}. 

    \begin{proof}[Proof of Theorem \ref{thm:trees}]
    Firstly if $T$ is a star the result follows from Lemma \ref{lem:star}, and so in what follows we will assume that $T$ is not a star. We claim that for any choice of root vertex $z$ of $T$ which is not a leaf, we have that 
        \begin{equation} \label{eq:trees_finalbutone}
        M_T(n)\leq 3+3\ceil{\height_z(T)/2}+\mu+\delta,
        \end{equation}
    with $\mu=\mu(T;z)$ and $\delta=\delta(T;z)$ as in \eqref{eq:trees_mu} and \eqref{eq:trees_delta}. Indeed suppose that $G_0$ is a choice of an $n$-vertex graph that maximises the  running time of the $T$-process $(G_i)_{i\geq 0}$. If the process has stabilised by time $i^* := 2+3\ceil{\height_z(T)/2}+\mu+\delta$, then \eqref{eq:trees_finalbutone} certainly holds. If not then, then by Proposition \ref{prop:tree_many_univ}, we have that 
    $G_{i^*}$ has a set $U^*$ of $|U|+\delta-2$ universal vertices. We will show that $G_{i^*+1}$ is complete and thus stable and so \eqref{eq:trees_finalbutone} holds. Indeed consider  some pair $x,y\in V(G_{i^*})$ of vertices that are non-adjacent vertices in $G_{i^*}$, then $xy$ completes a copy of $K_{1,\delta}$ with centre $x$ whose vertex set consists of $x$,$y$ and $\delta-1$ universal vertices.
    The remaining $|U|-1$ universal vertices together with $t-|U|-\delta$ vertices from $V(G)\setminus (U^*\cup\{x,y\})$ can be used to extend the copy of $K_{1,\delta}$ to a copy of $T$ by making the universal vertices together with $x$ play the role of the cover. As $x$ and $y$ were arbitrary, $G_{i^*+1}$ is a complete graph as required.

    The right hand side of \eqref{eq:trees_finalbutone} depends on the choice of root vertex $z$. We now show how to choose $z$ so that \eqref{eq:trees_finalbutone} implies the bound in Theorem \ref{thm:trees}. 
    Take a path $P$ of length $\diam(T)$ in $T$ and choose $z\in V(P)$ such that its distances to the endpoints of $P$ are $\floor{\diam(T)/2}$ and $\ceil{\diam(T)/2}$ (and hence $z$ is not a leaf).
    With this choice we have $\height_z(T) = \ceil{\diam(T)/2}$.
    The vertex cover $U=U(T;z)$ from Lemma \ref{lem:trees_vertexcover} contains at least $\ceil{\diam(T)/2}$ vertices from $P$.
    Any internal vertex of $P$ has at most two neighbours in $U\cap V(P)$ while the two endpoints of $P$ and any vertex in $V(T)\setminus V(P)$ each have at most one neighbour on $P$. Thus,
        \begin{linenomath} \begin{align*}
        \delta &= \min_{u\in U} |N_T(u)\setminus U| \leq \min_{u\in U\cap V(P)} |N_T(u)| \\
        &\leq \frac{|\{ (u,v) : u\in U\cap V(P), v\in V(T), uv\in E(T)\}|}{|U\cap V(P)|} \\
        &\leq \frac{2(\diam(T)-1)+t-(\diam(T)-1)}{\ceil{\diam(T)/2}} \\
        &\leq \frac{t-1}{\ceil{\diam(T)/2}} + 2.
        \end{align*} \end{linenomath}
    Combining this with \eqref{eq:trees_finalbutone} yields
        \begin{linenomath} \begin{equation}\label{eq:trees_final}
        M_T(n) \leq 3 + 3\ceil*{\frac{\diam(T)} 4} +  \frac{t-1}{\ceil{\diam(T)/2}} + 2 + \mu \leq \frac{29}{4}+\frac{3\diam(T)}{4}+\frac{2(t-1)}{\diam(T)}+\mu .
        \end{equation} \end{linenomath} 
    The diameter lies between $3$ and $t-1$ and $t\geq 4$ because $T$ is not a star, and we can bound $\mu$ by $t^2/8$ since the vertex cover number of any tree on $t$ vertices is at most $t/2$ and $M_{T[U]}(|U|) \leq \binom{|U|}2$.
    The right hand side of \eqref{eq:trees_final} is maximised when $\diam(T) = t-1$ (using here that $t\geq 4$). Therefore,
        \[
        M_T(n) \leq \frac{29}{4}+\frac{3(t-1)}{4}+2+ \frac{t^2}8\leq \frac 1 8 \cdot (t^2 + 6t + 68),
        \]
    as required. 
    \end{proof}

It remains to prove Proposition \ref{prop:tree_many_univ}. 
Before embarking on this,  we give some further useful tools for our proof. 
	
    \begin{obs}\label{obs:trees_core}
    Suppose $(G_i)_{i\geq 0}$ is some $T$-process for a tree $T$.     Let $T_0\subset G_i$ be a copy of $T$ for some $i\in\NN$, and let $x\in V(T_0)$ and $x'\in V(G)\setminus V(T_0)$. If $y\in N_{T_0}(x)$ is such that all $T_0$-neighbours of $x$ but $y$ are adjacent to $x'$ at time $i$, then $x'y \in E(G_{i+1})$ because it completes $T_0^{(x\to x')}$.
    \end{obs}

    \begin{figure}
    \centering \label{fig:switch}
    \includegraphics[width=0.7\linewidth]{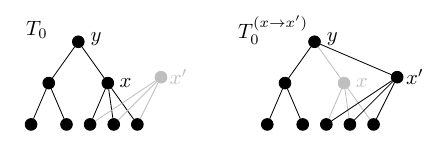}
    \caption{The situation of Observation \ref{obs:trees_core} for a concrete choice of $T_0$. On the left we have a copy $T_0$ of $T$ (black vertices and edges) at some time $i$. The vertex $x'$ is adjacent in $G$ to all $T_0$-neighbours of $x$ but $y$. On the right we see the copy $T_0^{(x\to x')}$ that is completed by $x'y$ at time $i+1$.}
    \label{fig:treeobservation}
    \end{figure}
	
See Figure \ref{fig:switch} for an example demonstrating Observation \ref{obs:trees_core}. The next lemma tells us that in a copy of $T$ the vertices corresponding to the elements of the cover $U$ given by Lemma \ref{lem:trees_vertexcover}, become \emph{almost universal} in the sense that they are adjacent to every vertex outside the copy of $T$.
	
    \begin{lem}\label{lem:trees_almostuniversal}
    Let $T$ be a tree with a root $z$ which is not a leaf and let $U=U(T;z)$ be the set of vertices of $T$ given by Lemma \ref{lem:trees_vertexcover}. Furthermore, let $(G_i)_{i\geq 0}$ be some $T$-process,   let $i_0\in\NN$ and let $T_0\subset G_{i_0}$ be a copy of $T$ with an isomorphism $\phi: T\to T_0$. Then at time $i_0+\ceil{\height_z(T)/2}$, every vertex in $\phi(U)$ is adjacent to every vertex in $V(G)\setminus V(T_0)$.
    \end{lem}
	
    \begin{proof}
    Let $w\in V(G)\setminus V(T_0)$. We show by induction on $i$ that $\phi(u)w\in E(G_{i_0+i})$ for every $0\leq i \leq \ceil{\height_z(T)/2}$ and every $u\in U$ with $\height_z(u) \in \{ 2i-1,2i \}$. This claim holds vacuously for $i=0$ because $U$ does not contain any leaves (cf. property (2) of Lemma \ref{lem:trees_vertexcover}).
	
    Given $u\in U$ with $\height_z(u) \in \{ 2i-1,2i \}$ where $i\geq 1$, there exists a child $v$ of $u$ that is not contained in $U$ by Lemma  \ref{lem:trees_vertexcover} (3).
    As $U$ is a cover all children of $v$ lie in $U$, so their images under $\phi$ are adjacent to $w$ at time $i_0+i-1$ due to the induction hypothesis. Observation \ref{obs:trees_core} with $x=\phi(v)$, $x'=w$, $y=\phi(u)$ implies $\phi(u)w\in E(G_{i_0+i})$. This completes the induction hypothesis and hence the proof of the claim.
    \end{proof}

We are now in a position to prove Proposition \ref{prop:tree_many_univ} and thus complete the proof of Theorem \ref{thm:trees}. 

    \begin{proof}[Proof of Proposition \ref{prop:tree_many_univ}]
    Fix $T$ with a root $z$ that is not a leaf and $U$, $\mu$ and $\delta$ as in the statement of the proposition. Let $(G_i)_{i\geq 0}$ be a  $T$-process  with $v(G_0)=n\geq 2t$ and suppose that it does not stabilise in the first $i^*$ steps (otherwise we are done), that is,  
    $E(G_{i+1})\setminus E(G_i)\neq\emptyset$ for $0\leq i \leq i^*$.
    Further, fix a copy $T_1\subseteq G_1$ of $T$, an isomorphism $\phi_1: T \to T_1$, and a set $U_1 := \phi_1(U)$.
    Let us rewrite $i^*$ in the following slightly more cumbersome form:
        \begin{equation} \label{eq:i*}
        i^* = \left( 1 + 2\ceil*{\frac{\height_z(T)}2} \right) + (\mu +1) + \left( 1 + \ceil*{\frac{\height_z(T)}2} \right) + 1+ (\delta-2).
                \end{equation}
    The summands on the right hand side correspond to different stages of the process. We will prove the following statements for each of these stages. 
        \begin{enumerate}[label=(\roman*)]
        \item \label{stage1} At time $i_1:=1+2\ceil*{\frac{\height_z(T)}2}$ there is a complete bipartite graph between $U_1$ and $V(G)\setminus U_1$.
        \item \label{stage2} After at most $\mu+1$ more steps we can find a copy $T_2$ of $T$ and an isomorphism $\phi_2:T\to T_2$ such that $\phi_2(U)\neq U_1$.
        \item \label{stage3} After at most $1+\ceil*{\frac{\height_z(T)}2}$ further steps, $U_1$ will be a set of universal vertices.
        \end{enumerate}
    If $\delta\leq 2$, then we  take $U_1$ as our collection of universal vertices. Note that when $\delta=1$, the contribution of  $\delta-2$ in \eqref{eq:i*} is negative, which is why we require the $+1$ term to cancel this out. In the case that $\delta>2$, we do one more stage as follows.
        \begin{enumerate}[label=(\roman*)] \addtocounter{enumi}{3}
        \item \label{stage4} Once there is a set of $|U|$ universal vertices, $\delta-2$ additional universal vertices will occur within at most $\delta-2$ steps.
        \end{enumerate}
    Note that once we have proved the claims \ref{stage1}-\ref{stage4}, the proof of the proposition will be complete as we will have found the required number of universal vertices in the first $i^*$ steps. We are  now going to prove the claims \ref{stage1}-\ref{stage4} sequentially.

    \vspace{2mm}

    \textbf{Stage \ref{stage1}:}
    At time $1+\ceil{\height_z(T)/2}$ every vertex in $U_1$ is adjacent to every vertex in $V(G)\setminus V(T_1)$ by Lemma \ref{lem:trees_almostuniversal}. It remains to show the existence of the edges between $U_1$ and $V(T)\setminus U_1$  and so fix some arbitrary choice of $u\in U_1$ and  $w\in V(T_1)\setminus U_1$. Moreover, fix some arbitrary $w'\in V(G)\setminus V(T_1)$, and note that $T_1^{(w\to w')}$ is contained in $G_{1+\ceil{\height_z(T)/2}}$. Indeed all of the edges adjacent to $w$ in $T_1$ must have their other endpoint in $U_1$ (as $U$ is a cover) and all vertices in $U_1$ are adjacent to $w'$ at time $1+\ceil{\height_z(T)/2}$. 
    Lemma \ref{lem:trees_almostuniversal} applied to $T_1^{(w\to w')}$ then guarantees that in  $G_{1+2\ceil{\height_z(T)/2}}$, $u$ and $w$ are adjacent and so, as $u$ and $w$ were arbitrary,  the sets $U_1$ and $V(G)\setminus U_1$ are the partite sets of a complete bipartite graph.

    \vspace{2mm}

    \textbf{Stage \ref{stage2}:}
    For $i\geq i_1+1$ all edges in $E(G_i)\setminus E(G_{i-1})$ have either both their endpoints in $U_1$ or both their endpoints in $V(G)\setminus U_1$.	
    We want to choose a copy $T_2$ of $T$ and an isomorphism $\phi_2: T\to T_2$ such that $T_2$ is completed by an edge $e_2$ at time $i_2$ for some $i_2\in\NN$ with $i_1+1 \leq i_2 \leq i_1+1 + \mu$ and $\phi_2(U) \neq U_1$. 
    Suppose that for every $i_2$ in the above range and every choice of $T_2$, $\phi_2$, $e_2$ one has $\phi_2(U) = U_1$.
    In that case $e_2$ has an endpoint in $U_1$ and hence must be fully contained in $U_1$. 
    Then $T_2[U_1]$ is a copy of $T[U]$ in $G_{i_2}$ with $T_2[U_1]-e_2 \subset G_{i_2-1}$.
    Moreover, any copy of $T[U]-e$ for some $e\in E(T[U])$ in $G_{i_2-1}[U_1]$ can be extended to a copy of $T-e$ using an arbitrary set of $t-|U|$ vertices from $V(G)\setminus U_1$.
    Therefore, the graphs $G_{i_1+i}[U_1]$, $0\leq i \leq \mu+1$, are the first $\mu+2$ elements (including the starting graph)  of the $T[U]$-process on $G_{i_1}[U_1]$.
    This however contradicts the definition of $\mu$   \eqref{eq:trees_mu} as no $T[U]$-process on $|U|$ many vertices can last longer than $\mu$ steps before stabilising. Hence there must be some choice of  $i_2$ in the desired range as well as $T_2$,  $\phi_2$ and $e_2$ such that $e_2$ lies on the vertices $V(G)\setminus U_1$ and $\phi_2(U)\neq U_1$.  We fix such  a choice of $i_2,T_2,\phi_2$ and $e_2$. 

    \vspace{2mm}

    \textbf{Stage \ref{stage3}:}
    Pick $u\in U$ such that $\phi_2(u) \notin U_1$ and $\height_z(u)$ is minimised among all vertices in $U\setminus\phi_2^{-1}(U_1)$.
    Let $v$ be a child of $u$ in $T$ that does not lie in $U$ which exists by Lemma \ref{lem:trees_vertexcover} (3).
    By the minimality of $\height_z(u)$, the children of $\phi_2(v)$, if there are any, lie in $U_1$ and thus are adjacent to all vertices in $V(G)\setminus U_1$ due to stage \ref{stage1}.
    Consequently, Observation \ref{obs:trees_core} with $x=\phi_2(v)$, $x'=w$ for $w\in V(G)\setminus (U_1\cup V(T_2))$ and $y=\phi_2(u)$ implies $T_2^{(\phi_2(v)\to w)} \subset G_{i_2+1}$ for every $w\in V(G)\setminus (U_1\cup V(T_2))$ so $\phi_2(u)$ is adjacent to every vertex in $V(G)\setminus (U_1\cup V(T_2))$  at time $i_2+1$.
    Choose an arbitrary set $W \subset V(G)\setminus (U_1\cup V(T_2))$ of size\footnote{This is possible because $n\geq 2t$.} $t-|U_1|$ and a map $\phi_3: V(T) \to U_1 \cup W$ such that $\phi_3(u) = \phi_1(u)$ for all $u\in U$.
    Since $V(T)\setminus U$ is an independent set and all edges between $U_1$ and $V(G)\setminus U_1$ are present at time $i_2$, in particular those between $\phi_2(u)$ and $U_1$, $T_3 := \phi_3(T)$ is a copy of $T$ in $G_{i_2+1}$ that lies in the $G_{i_2+1}$-neighbourhood of $\phi_2(u)$.
    Then for every $u' \in U_1$, we can replace $u'$ by $\phi_2(u)$ in $T_3$ to obtain another copy of $T$, that is, $T_3^{(u'\to \phi_2(u))} \subset G_{i_2+1}$.
    Now Lemma \ref{lem:trees_almostuniversal} implies that at time $i_2+1+\ceil{\height_z(T)/2}$ every $u'\in U_1$ is adjacent to every vertex in $\phi_3^{(u'\to \phi_2(u))}(U) = U_1\setminus\{ u' \} \cup \{ \phi_2(u) \}$ and hence is a universal vertex.
    Recall that $i_2+1+\ceil{\height_z(T)/2} \leq 3 + 3\ceil{\height_z(T)/2} + \mu$.
    We have shown that $U_1$ is a set of universal vertices at time $3 + 3\ceil{\height_z(T)/2} + \mu$.

    \vspace{2mm}

    \textbf{Stage \ref{stage4}:}
    The remaining $\delta-2$ universal vertices will be obtained by applying the next claim $(\delta-2)$ times.

        \begin{clm}\label{clm:trees_finaliteration}
        Let $i \in\NN$, $0\leq \delta^*<\delta-2$ and assume that $G_i$ has precisely $|U|+\delta^*$ universal vertices. Then there exists $v\in V(G)$ which is universal in $G_{i+1}$ but not in $G_i$.
        \end{clm}
        \begin{proof}
        Let $U^*$ be the set of universal vertices in $G_i$.
        Any vertex in $V(G)\setminus U^*$ with at least $\delta-\delta^*-2$ $G_i$-neighbours outside $U^*$ is the centre of a copy of $K_{1,\delta-1}$ using at most $\delta^*+1$ vertices from $U^*$. Therefore such a vertex  will be universal in $G_{i+1}$ by the definition of $\delta$ \eqref{eq:trees_delta}.
        We now show that such a vertex exists. 
        Suppose it did not, i.e. $|N_{G_i}(v) \setminus U^*|< \delta-\delta^*-2$ for each $v\in V(G)\setminus U^*$. Take a copy $T_0\subset G_{i+1}$ of $T$ which is missing precisely  one edge in $G_{i}$  and fix an isomorphism $\phi : T \to T_0$. The set $\phi(U)\setminus U^*$ is non-empty because $U$ is a vertex cover of $T$ and a universal vertex of $G_i$ cannot receive a new neighbour at time $i+1$. Every $v\in \phi(U)\setminus U^*$ satisfies
            \[
            \left|N_{T_0}(v) \cap U^*\setminus \phi(U)\right| = \left|N_{T_0}(v)\setminus \phi(U)\right|- \left|\left(N_{T_0}(v)\setminus\phi(U)\right)\setminus U^*\right|\geq \delta - (\delta-\delta^*-2) = \delta^* + 2,
            \]
            using here the definition of $\delta$ \eqref{eq:trees_delta}.
        Thus, summing over $v\in \phi(U)\setminus U^*$, we get that
            \[
            (2+\delta^*)\cdot |\phi(U)\setminus U^*| \leq \left| E_{T_0}\left(\phi(U)\setminus U^*, U^*\setminus \phi(U)\right) \right|
            \]
        and, since $T_0$ is a tree, 
            \[
            \left| E_{T_0}(\phi(U)\setminus U^*, U^*\setminus \phi(U)) \right| \leq |\phi(U)\setminus U^*| + |V(T_0) \cap U^*\setminus \phi(U)| - 1.
            \]
        Combining the last two estimates results in
            \begin{linenomath} \begin{equation}\label{eq:trees_edgecounting}
            (1+\delta^*)\cdot |\phi(U)\setminus U^*| \leq |V(T_0) \cap U^*\setminus \phi(U)| - 1.
            \end{equation} \end{linenomath} 
        However, we also have
            \[
            |V(T_0) \cap U^*\setminus \phi(U)| + |\phi(U)\cap U^*| \leq |U^*| =|U|+\delta^*
            \]
        and hence
            \begin{linenomath} \begin{equation}\label{eq:trees_blockedplaces}
            |V(T_0) \cap U^*\setminus \phi(U)|\leq |U|- |\phi(U)\cap U^*| + \delta^* = |\phi(U)\setminus U^*|+\delta^*.
            \end{equation} \end{linenomath} 
        The inequalities \eqref{eq:trees_edgecounting} and \eqref{eq:trees_blockedplaces} together imply
            \[
            (1+\delta^*)\cdot |\phi(U)\setminus U^*| < |\phi(U)\setminus U^*|+\delta^*,
            \]
        which is a contradiction as $\phi(U)\setminus U^*$ is non-empty. Therefore we indeed have a vertex in $V(G)\setminus U^*$ with at least $\delta-\delta^*-2$ $G_i$-neighbours outside $U^*$ and so a new universal vertex at time $i+1$. 
        \end{proof}

    As $U_1$ consists of universal vertices at time $3+3\ceil*{\height_z(T)/2}+\mu$, Claim \ref{clm:trees_finaliteration} guarantees the existence of at least $|U|+\delta-2$ universal vertices in $G_{1+3\ceil*{\height_z(T)/2}+\mu+\delta}$, completing the proof of Proposition \ref{prop:tree_many_univ}. 
    \end{proof}

\section{Small degrees are necessary for sublinear running time }\label{sec:min2max3}
In this section, we will prove Theorem \ref{thm:23bound}.
Before embarking on the proof we give the following simple lemma.  

    \begin{lem}\label{lem:growing_clique}
    Let $H$ be a graph, $r:=v(H)-1$ and $\delta:=\delta(H)$. Further, suppose $n\geq r+1$ and  $G$ is an $n$-vertex graph with an ordering $v_1,\ldots,v_n$ of its vertices such that $\{ v_1,\ldots,v_r \}$ is a clique and for every $i\in [r+1,n]$, $|\{ j \in [i-1] : v_jv_i \in E(G) \}|\geq \delta-1$, i.e. every vertex but the first $r$ is adjacent to at least $\delta-1$ vertices preceding it.
    Then in the $H$-process $(G_i)_{i\geq 0}$ on $G$, for every $i\geq 0$, $\{ v_1,\ldots, v_{r+i} \}$ is a clique in $G_i$. In particular, $G_{n-r}=K_n$. 
    \end{lem}
    \begin{proof}
    We induct on $i\geq 0$. The base case is given by the definition of $G=G_0$. For $i\geq 1$, suppose the statement holds up to time $i-1$, and so $G_{i-1}$ hosts a clique $K$ on vertices $\{ v_1,\ldots, v_{r+i-1} \}$. As $v_{r+i}$ has $\delta-1$ neighbours in $K$, any edge  missing between $v_{r+i}$ and $K$ in $G_{i-1}$ can be used to complete a copy of $H$ with $v_{r+i}$  being a vertex of degree $\delta$. Hence in $G_i$ there is a clique on vertices $\{ v_1,\ldots, v_{r+i} \}$, completing the induction step and the proof. 
    \end{proof}

Lemma \ref{lem:growing_clique} will be useful for us in giving a lower bound on the running time of a $H$-process. Indeed, we will use a starting graph $G$ satisfying the hypothesis of the lemma and show that the  clique number increases by at most a constant in each step of the $H$-process. Lemma \ref{lem:growing_clique} guarantees that the process will eventually percolate, that is, that the final graph will be $K_n$. These two facts together imply a linear running time. The details follow. 

    \begin{proof}[Proof of Theorem \ref{thm:23bound}]
    Let $H$ be a connected graph with $\delta:=\delta(H)\geq 2$ and $\Delta(H)\geq 3$ as in the statement of the theorem. Further, fix $r:=v(H)-1$ and for $n$ sufficiently large, we fix $\ell := \floor{\frac{n-r}{(\delta-1)^r}}$ and $n'=r+\ell\cdot(\delta-1)^r\leq n$. 

    We will build an $n'$-vertex starting graph $G$ which will be the union of a $K_r$, a graph $G'$ which we will define in a moment and a complete bipartite graph between the vertices of $K_{r}$ and some vertices of $G'$.  Let $\ZZ_t=\{1,\ldots,t\}$ denote the cyclic group with $t$ elements and denote the standard basis vectors of $\ZZ_{\delta-1}^{r}$ by $e_1,\ldots,e_r$. For $r < j \leq \ell$, let $e_j := e_{j \bmod r}$. 
    Let $G'$ be given by
        \begin{linenomath}\begin{equation*}
        V(G') = [\ell] \times \ZZ_{\delta-1}^{r}, \qquad
        E(G') = \left\lbrace \{ (j,x) \, , \, (j+1, x +\lambda e_j) \} : j \in [\ell-1], x\in\ZZ_{\delta-1}^{r}, \lambda\in \ZZ_{\delta-1} \right\rbrace.
        \end{equation*} \end{linenomath}
    It is the union of $\ell-1$ pairwise isomorphic $(\delta-1)$-regular bipartite graphs with partite sets $\{ j \} \times \ZZ_{\delta-1}^{r}$ and $\{j+1\} \times \ZZ_{\delta-1}^{r}$ for $1\leq j \leq \ell-1$. 

        \begin{clm}\label{clm:min2max3_free}
        For any $e\in E(H)$ and any connected component $H'$ of $H-e$, $G'$ is $H'$-free.
        \end{clm}
        \begin{proof}
        Let $e\in E(H)$, and let $H'$ be a connected component of $H-e$. There must be a cycle in $H'$.
        If not, $H'$ would be a tree and would thus have at least two leaves. Since $H-e$ can have at most two vertices of degree $\delta-1$, $H'$ would have precisely two leaves, that is, $H'$ would be a path. The endpoints of $H'$ would be the endpoints of $e$ because $\delta = 2$. But then $H' + e $ would be a cycle, so $H$ would have maximum degree two, which contradicts our assumptions $\Delta(H)\geq 3$.
    
        If $\delta = 2$, $G'$ is just a path of length $\ell-1$ so in this case it is clearly $H'$-free.
        For the rest of the proof of Claim \ref{clm:min2max3_free} we assume that $\delta \geq 3$.
        Suppose there was a copy of $H'$ in $G'$. Denote that copy  by $H_0'$. Let
            \begin{linenomath}\begin{align*}
            j_0 &:= \min \left\lbrace j \in [\ell] : V(H_0') \cap (\{ j \}\times \ZZ_{\delta-1}^r) \neq \emptyset \right\rbrace, \\
            j_1 &:= \max \left\lbrace j \in [\ell] : V(H_0') \cap (\{ j \}\times \ZZ_{\delta-1}^r) \neq \emptyset \right\rbrace,
            \end{align*} \end{linenomath}
        and for $j\geq j_0$, define 
            \[
            H_0'[j_0,j] := H_0'\left[V(H_0')\cap ([j_0,j]\times\ZZ_{\delta-1}^r)\right].
            \]
        There exists a path from a vertex in $\{j_0\}\times\ZZ_{\delta-1}^r$ to a vertex $\{j_1\}\times\ZZ_{\delta-1}^r$ in $H_0'$.
        Such a path must intersect $\{ j \}\times \ZZ_{\delta-1}^r$ for $j_0 \leq j \leq j_1$.
        Thus $j_1 - j_0 \leq r$.
        Since $H-e$ has at most two vertices of degree $\delta-1$ we get
            \[
            \left| 	V(H_0')\cap (\{ j_0 \}\times\ZZ_{\delta-1}^r)  \right|= 1 = \left| V(H_0')\cap (\{ j_1 \}\times\ZZ_{\delta-1}^r) \right|.
            \]
        Take the unique $x_0,x_1 \in \ZZ_{\delta-1}^r$ with $(j_0,x_0),(j_1,x_1)\in V(H_0')$. 

        Now we show that for every $j\in [j_0,j_1]$, $H_0'[j_0,j]$ must be a full $(\delta-1)$-ary tree with root $(j_0,x_0)$ whose set of leaves is $V(H_0')\cap (\{ j \}\times\ZZ_{\delta-1}^r)$.
        We induct on $j\in [j_0,j_1]$.
        As $\left| V(H_0')\cap (\{ j_0 \}\times\ZZ_{\delta-1}^r)  \right|= 1$, $H_0'[j_0,j_0]$ is just an isolated vertex and $H_0'[j_0,j_0+1]$ is a star with $\delta-1$ leaves.
        Let $j_0+1 < j \leq j_1$. Every vertex $(j-1,x)$ in $V(H_0')\cap (\{ j-1 \}\times\ZZ_{\delta-1}^r)$ has precisely one $H_0'[j_0,j-1]$-neighbour by the induction hypothesis.
        Consequently, the remaining $H_0'$-neighbours of $(j-1,x)$ must lie in $(\{ j \}\times\ZZ_{\delta-1}^r)$. There are at least $\delta-1$ of them so the $H_0'$-neighbours of $(j-1,x)$ in $\{ j \}\times\ZZ_{\delta-1}^r$ are precisely its $\delta-1$ $G'$-neighbours in $\{ j \}\times\ZZ_{\delta-1}^r$. It remains to prove that the neighbourhoods of any two distinct $(j-1,x), (j-1,y) \in V(H_0')$ are disjoint.
        The unique paths from $(j-1,x)$ and $(j-1,y)$ to the root $(j_0,x_0)$ in $H_0'[j_0,j-1]$ amount to $\lambda_{j_0},\ldots,\lambda_{j-1},\mu_{j_0},\ldots,\mu_{j-1}\in \ZZ_{\delta-1}$ with $(\lambda_{j_0},\ldots,\lambda_{j-1}) \neq (\mu_{j_0},\ldots,\mu_{j-1})$ and 
            \[
            x = x_0 + \lambda_{j_0} e_{j_0} +\ldots+ \lambda_{j-1} e_{j-1} \qquad,\qquad y = x_0 + \mu_{j_0} e_{j_0} +\ldots+ \mu_{j-1} e_{j-1}.
            \]
        Thus for every $z\in\ZZ_{\delta-1}^r$ with $x-z = \lambda_j e_j$ for some $\lambda_j\in\ZZ_{\delta-1}$ one has
            \[
            y-z = y-x + x-z = (\mu_{j_0}-\lambda_{j_0}) e_{j_0} +\ldots+ (\mu_{j-1}-\lambda_{j-1}) e_{j-1} +  \lambda_j e_j \notin \ZZ_{\delta-1}\cdot e_j.
            \]
        Here we used that $j-j_0 \leq j_1 - j_0 \leq r$ so $e_{j_0},\ldots,e_j$ are linearly independent.
        This completes the induction.
	
        As $H_0'=H_0'[j_0,j_1]$, we have arrived at a contradiction since trees have minimum degree one while $\delta(H')\geq\delta-1\geq 2$. Therefore $G'$ is $H'$-free.
        \end{proof}

    We now define the $n'$-vertex starting graph $G$ by
        \begin{linenomath}\begin{align*}
        V(G) &= \{ v_1,\ldots,v_r \} \ \cup \  V(G'), \\
        E(G) &= \binom{\{ v_1,\ldots,v_r \}}2 \cup E(G') \cup \left\lbrace \{ v_j, (1,x) \} : j\in[r], x\in\ZZ_{\delta-1}^r \right\rbrace,
        \end{align*} \end{linenomath}
    where $v_1,\ldots,v_r$ are $r$ newly introduced vertices. Let $(G_i)_{i\geq 0}$ be the $H$-process on $G=G_0$. Pick an ordering of the vertices of $V(G)$ such that $v_1,\ldots,v_r$ are the first $r$ vertices and for $j\in[\ell-1]$ each vertex in $\{ j \}\times\ZZ_{\delta-1}^r$ precedes each vertex in $\{ j+1 \}\times\ZZ_{\delta-1}^r$. With such an ordering $G$ satisfies the hypotheses of Lemma \ref{lem:growing_clique}.
    Write
        \begin{linenomath}\begin{align*}
        G'[j,\ell] &:= G'\left[V(G')\cap ([j,\ell]\times\ZZ_{\delta-1}^r)\right] \;\text{ and } \\
        G_i[j,\ell] &:= G_i\left[V(G)\cap ([j,\ell]\times\ZZ_{\delta-1}^r)\right]
        \end{align*} \end{linenomath}
    for any $j\in[\ell]$, $i\geq 0$, and let $i_j$ be the smallest positive integer such that $E(G_{i_j})\setminus E(G')$ contains an edge touching $[j,\ell]\times\ZZ_{\delta-1}^r$. The $i_j$ are non-decreasing and well-defined as Lemma \ref{lem:growing_clique} guarantees that every vertex of $G$ receives a new neighbour at some stage of the process.
    We claim that $i_j > i_{j-r}$ for all $j \in [r+1,\ell]$. Suppose there existed $j\in [r+1,\ell]$ with $i_j = i_{j-r}$. At time $i_j-1$ we can find a copy $H'$ of $H-e$ for some $e\in E(H)$ with a vertex $w$ in $[j,\ell]\times \ZZ_{\delta-1}^r$. 
    Since $i_j = i_{j-r}$, there are no edges in $E(G_{i_j-1})\setminus E(G')$ with an endpoint in $[j-r,\ell]\times\ZZ_{\delta-1}^r$, hence $G_{i_j-1}[j-r,\ell] = G'[j-r,\ell]$. Note that the vertices that can be reached from $[j,\ell]\times\ZZ_{\delta-1}^r$  by paths of length at most $r$ in $G'$ is precisely the vertex set $[j-r,\ell]\times\ZZ_{\delta-1}^r$. But then $H'\cap G_{i_j-1}[j-r,\ell]$ contains the connected component of $H'$ containing $w$ and so there is a copy of a connected component of $H-e$ contained in $G'$.
    This contradicts Claim \ref{clm:min2max3_free}.

    Recall that the order of $G$ is $n'=r+\ell\cdot (\delta-1)^r$. Therefore
        \[
        M_H(n') \;\geq\; \tau_H(G) \;\geq\; i_\ell \;\geq\; i_{r\cdot\floor{\ell/r}} \;\geq\; \sum_{s=1}^{\floor{\ell/r}-1} \left(i_{(s+1)r}-i_{sr}\right) \;\geq\; \floor*{\frac \ell r} - 1.
        \]
    Note that during  the $H$-process on any graph no isolated vertex receives a new neighbour since $\delta \geq 2$.
    Therefore, $M_H(n)$ is non-decreasing in $n$ and
        \[
        M_H(n) \geq M_H(n') \geq \floor*{\frac \ell r} - 1 = \Omega(n).
        \]
    The `moreover' of Theorem \ref{thm:23bound} states that if $H$ is bipartite we may obtain the linear lower bound by choosing a bipartite starting graph. Observe that the graph $G'$ is already bipartite. 
    In order to replace $G$ by a bipartite starting graph it will be sufficient to replace the $r$-clique $\{ v_1,\ldots,v_r \}$ by a complete bipartite graph with partite sets of size $r$ each.
    The following is the bipartite analogue of Lemma \ref{lem:growing_clique}.

        \begin{clm}\label{clm:growing_complete_bipartite_graph}
        Let $G$ be a bipartite graph with partite sets $X,Y$. If $v_1,\ldots,v_n$ is an ordering of $V(G)$ such that $\{ v_1,\ldots,v_r \}\subseteq X$, $\{ v_{r+1},\ldots,v_{2r} \}\subseteq Y$ and for every $i > 2r$, either
            \[
            |\{ j \in [i-1] : v_j\in Y, v_iv_j\in E(G) \}|\geq \delta-1 \qquad\text{ or }\qquad |\{ j \in [i-1] : v_j\in X, v_jv_i\in E(G) \}|\geq \delta-1    
            \]
        then at time $i$, $\{ v_1,\ldots,v_{2r+i} \}$ is the vertex set of a (not necessarily induced) complete bipartite graph with partite sets $X\cap\{ v_1,\ldots,v_{2r+i} \}$ and $Y\cap\{ v_1,\ldots,v_{2r+i} \}$.
        \end{clm}
	
    Claim \ref{clm:growing_complete_bipartite_graph} follows from an inductive application of the observation that in the $H$-process on $G$, any vertex with at least $\delta-1$ neighbours in a partite set $S$ of a complete bipartite graph with at least $r$ vertices in each part will be adjacent to all remaining vertices in $S$ after one more step.

    Define the bipartite starting graph $G$ via
        \begin{linenomath}\begin{align*}
        V(G) &= \{ v_1,\ldots,v_{2r} \} \ \cup \  V(G') \;\text{ and }\\
        E(G) &= \left\lbrace v_iv_j : i\in[r], j\in[r+1,2r] \right\rbrace\cup E(G') \cup \left\lbrace \{ v_j, (1,x) \} : j\in[r+1,2r], x\in\ZZ_{\delta-1}^r \right\rbrace,
        \end{align*} \end{linenomath} 
    with $G'$ as defined earlier in the proof. 
    Again, choose an ordering of $V(G)$ such that $v_1,\ldots,v_{2r}$ are the first $2r$ vertices and the vertices in $\{ j \}\times\ZZ_{\delta-1}^r$ precede the vertices in $\{ j+1 \}\times\ZZ_{\delta-1}^r$. 
    Then the hypotheses of Claim \ref{clm:growing_complete_bipartite_graph} are satisfied with
        \[
        X = \{ v_1,\ldots,v_r \} \cup \bigcup_{j \in [\ell] \text{ odd}} \{ j \}\times\ZZ_{\delta-1}^r
        ,\qquad\qquad 
        Y = \{ v_{r+1},\ldots,v_{2r} \} \cup \bigcup_{j \in [\ell] \text{ even}} \{ j \}\times\ZZ_{\delta-1}^r.
        \]
    The rest of the proof does not differ from the proof of the first part of Theorem \ref{thm:23bound} after Lemma \ref{lem:growing_clique} was invoked.
    \end{proof}

\section{Complete bipartite graphs with a part of size 2}\label{sec:k2s_upper}
In this section we prove Proposition \ref{prop:k2s_upper}, establishing the asymptotic maximum running time of the $K_{2,s}$-process.  

    \begin{proof}[Proof of Proposition \ref{prop:k2s_upper}]
    Fix $s\geq 3$ and note that Theorem \ref{thm:23bound} implies that $M_{K_{2,s}}(n)=\Omega(n)$ as $H=K_{2,s}$ has minimum degree 2 and maximum degree $s\geq 3$. Therefore it remains to prove the upper bound and so we fix $G$ to be an arbitrary $n$-vertex graph and let $(G_i)_{i\geq 0}$ be the $K_{2,s}$-process on $G$. We will show that $\tau_H(G)=O(n)$, and so $M_H(n)=O(n)$ as $G$ is arbitrary. We begin with some simple claims. Here and throughout this proof, for a vertex $v\in V(G)$  and $0\leq i \in \NN$, we define the shorthand $N_i(v):=N_{G_i}(v)$ to denote the neighbourhood of vertex $v$ at time $i$. 

        \begin{clm}\label{clm:k2s_nbr}
        For every $0\leq i\in \mathbb{N}$, if two vertices $x,y\in V(G)$ have at least $s-1$ common neighbours in $G_i$, then $N_i(x)\setminus\{y\} \subseteq N_{i+1}(y)$.
        \end{clm}
        \begin{proof}
        As $N_i(y)$ is non-decreasing in $i$ it suffices to show that $N_i(x)\setminus (N_i(y)\cup\{y\}) \subseteq N_{i+1}(y)$.
        Fix $s-1$ common neighbours $v_1,\ldots,v_{s-1}$ of $x$ and $y$ at time $i$, and let $v\in N_i(x)\setminus (N_i(y)\cup\{y\})$. Then $x,y$ and $v_1,\ldots,v_{s-1},v$ form a copy of $K_{2,s}-e$ (with $e$ any edge of $K_{2,s}$) in $G_i$ that can be extended to a copy of $K_{2,s}$ by adding the edge $yv$.
        \end{proof}

    Our next claim shows that many edges can be added quickly   in the $K_{2,s}$-process. 

        \begin{clm}\label{clm:k2s_blocks1}
        Let $0\leq i\in \NN$ and $A,B\subseteq V(G)$ be a pair of disjoint sets such that any two vertices from the same set have at least $s-1$ common neighbours at time $i$. If $E_{G_i}(A,B)\neq\emptyset$, then $E_{G_{i+2}}(A,B) = \{ xy : x \in A, y\in B \}$. Similarly, if $C\subseteq V(G)$ is a single set such that every pair of vertices have at least $s-1$ common neighbours at time $i$ and $G_i[C]$ is non-empty, then the vertices of $C$ host a clique at time $i+2$. 
        \end{clm}
        \begin{proof}
        Let $x_0\in A$, $y_0\in B$ with $x_0y_0\in E(G_i)$. Any $x\in A$ with $x\neq x_0$ shares at least $s-1$ neighbours with $x_0$. Hence $xy_0\in E(G_{i+1})$ by Claim \ref{clm:k2s_nbr}. Similarly $x_0y\in E(G_{i+1})$ for every $y\in B \setminus\{y_0\}$. Applying Claim \ref{clm:k2s_nbr} again we obtain $N_{i+2}(x) \supseteq N_{i+1}(x_0)\setminus\{x\} \supseteq B$ for each $x\in A$ as well as $N_{i+2}(y) \supseteq N_{i+1}(y_0)\setminus\{y\} \supseteq A$ for every $y\in B$.
        The proof of the second statement is almost identical, we leave the details to the reader.
        \end{proof}

    Given a \emph{partition} $\cP$ of the vertex set $V(G)$ into $\cP=\{A_1,\ldots, A_k\}$, we refer to the parts $A_i$ as \emph{blocks}. The idea of our proof is to have an evolving partition throughout the process $(G_i)_{i\geq 0}$, maintaining that at each step, each block of the partition has the property that any pair of vertices in the block share at least $s-1$ common neighbours. Our final claim shows that if the process does not terminate, then within 4 steps, we can define a coarser partition with the same property. 

        \begin{clm}\label{clm:k2s_blocks2}
        Suppose $0\leq i\leq \tau_{K_{2,s}}(G)-4$ and $\cP$ is a partition of $V(G)$ such that any two vertices from the same block have at least $s-1$ common neighbours at time $i$.  Then there exist $A,B\in\cP$ for which $\cP\setminus\{ A,B \}\cup\{ A\cup B \}$ is a partition of $V(G)$ such that any two vertices from the same block have at least $s-1$ common neighbours at time $i+4$.
        \end{clm}
        \begin{proof}
        Let $e\in E(G_{i+3})\setminus E(G_{i+2})$ be added at time $i+3$ (which is possible as $i\leq \tau_{K_{2,s}}(G)-4$) and let $A',B'$ be the (not necessarily distinct) blocks of $\cP$ containing the endpoints of $e$. Then $E_{G_i}(A',B')=\emptyset$ for otherwise $e\in E(G_{i+2})$ by Claim \ref{clm:k2s_blocks1}. This tells us that
            \[ 
            i^* := \min \{ i' \geq 0 : E_{G_{i'}}(A',B') \neq\emptyset \} \in \{ i+1,i+2,i+3 \}. 
            \] 
        Fix a copy $K$ of $K_{2,s}$ that was completed by an edge in $E_{G_{i^*}}(A',B')$ at time $i^*$. Let $x_0,y_0$ be the two vertices forming its partite set of size two. These two vertices cannot lie in the same block since there were no edges from $A'$ to $B'$ at time $i^*-1$. Denote the block containing $x_0$ by $A$ and the block containing $y_0$ by $B$.
        For every $x\in A$, $y\in B$ we have by Claim \ref{clm:k2s_nbr}, that 
            \begin{equation}\label{eq:k2s_blocks} 
            N_{i+4}(x) \cap N_{i+4}(y) \supseteq V(K)\setminus\{x_0,y_0,x,y\}.
            \end{equation}
        The set on the right hand side of \eqref{eq:k2s_blocks} has size at least $s-1$ if $\{x,y\}\not\subseteq V(K)\setminus\{x_0,y_0\}$, and size $s-2$ if both $x$ and $y$ lie in $V(K)\setminus\{x_0,y_0\}$.
        In the second case $x_0$ and $y_0$ lie in $N_{i+4}(x) \cap N_{i+4}(y)$, which tells us that $\left| N_{i+4}(x) \cap N_{i+4}(y) \right| \geq s$.
        Hence the partition $\cP\setminus\{ A,B \}\cup\{ A\cup B \}$ has the desired property.
        \end{proof}

    Set $\cP_0 := \{ \{ v \} : v\in V(G) \}$. This partition trivially fulfils the condition that any two vertices from the same block have at least $s-1$ common neighbours at time $0$. We may inductively apply Claim \ref{clm:k2s_blocks2} to obtain a sequence $\cP_0,\cP_4,\ldots,\cP_\ell$ of partitions of $V(G)$ such that $\ell\in 4\cdot\NN_0$,  
        \begin{linenomath}\begin{equation*} 
        \ell \geq \tau_{K_{2,s}}(G) - 3
        \end{equation*}\end{linenomath}
    and for all $i\in\{ 4,8,\ldots,\ell \}$, $\cP_i$ is a partition of $V(G)$ that is such that any two vertices from the same block have at least $s-1$ common neighbours at time $i$. Moreover $\cP_i$ has fewer blocks than $\cP_{i-4}$ by construction and so there can be at most $n$ distinct partitions in this process. This implies
        \[ 
        n \geq \frac \ell 4 \geq \frac{\tau_{K_{2,s}}(G) - 3} 4,
        \]
    or equivalently, $\tau_{K_{2,s}}(G)  \leq 4n+3$, as required. 
    \end{proof}

\section{Constructions via simulation}\label{sec:counterexample}
In this section we use \emph{simulations} to give constructions of graph processes. The following definition captures the key idea of  our method. 

    \begin{dfn} \label{def:simulation}
    Suppose $(\tilde G_i)_{i\geq 0}$ is some $\tilde H$-process on a starting graph $\tilde G$. We say a $H$-process $(G_i)_{i\geq 0}$ on a starting graph $G$ \emph{simulates} the  $\tilde H$-process $(\tilde G)_{i\geq 0}$ if 
        \begin{enumerate}
        \item \label{item:sim 1} $\tilde G\subseteq G$ as an induced subgraph, and 
        \item \label{item:sim 2} for $i\geq 0$, we have\footnote{Here and throughout this section we consider each $\tilde G_i$ with $i\geq 0$ as a graph on the vertex set $V(\tilde G)\subseteq V(G)$.} $E(G_i)\setminus E(G)=E(\tilde G_i)\setminus E(\tilde G)$. 
        \end{enumerate}
    \end{dfn}

Note that the second condition in Definition \ref{def:simulation} gives that for $i\geq0$, we have that $G_i[V(\tilde G)]=\tilde G_i$, that is, if we focus solely on the vertices of $\tilde G$ in the $H$-process $(G_i)_{i\geq 0}$, what we see is the $\tilde H$-process $(\tilde G)_{i\geq 0}$. In fact \eqref{item:sim 2} implies that these edges on $V(\tilde G)$ are the \emph{only} edges added throughout the process $(G_i)_{i\geq 0}$.  In particular, this implies the following simple observation. 

    \begin{obs} \label{obs:simulate}
    If  the $H$-process $(G_i)_{i\geq 0}$ on $G$ \emph{simulates} the  $\tilde H$-process $(\tilde G)_{i\geq 0}$ on $\tilde G$, then $\tau_H(G)= \tau_{\tilde H}(\tilde G)$.
    \end{obs}

One way of  generating examples of simulation is by considering disjoint unions of graphs. Indeed, for $H=H'\sqcup \tilde H$ and some $\tilde H$-process $(\tilde G_i)_{i\geq 0}$ on a graph $\tilde G$, we can consider the $H$-process $(G_i)_{i\geq 0}$ on $G=H'\sqcup \tilde G$. Under certain conditions one can show that the copy of $H'$ is `locked in' in the $H$-process $(G_i)_{i\geq 0}$, that is, the role of $H'$ in any copy of $H$ is always played by the copy of $H'$ in $G_i$ that is disjoint from the vertices of $\tilde G\subseteq G$. Therefore the only edges added are those corresponding to copies of $\tilde H$ in the $\tilde H$-process on $\tilde G$ and so the $H$-process $(G_i)_{i\geq 0}$ on $G$ indeed simulates the  $\tilde H$-process $(\tilde G)_{i\geq 0}$ on $\tilde G$. One key  condition that is needed for this to work is that  copies of $H'$ minus an edge do not appear in any of the $\tilde G_i$,

Our constructions will build on this idea of having some subgraph $H'\subseteq H$ which is locked in place throughout the process $(G_i)_{i\geq 0}$ and so we need graph processes $(\tilde G_i)_{i\geq 0}$ which never contain some copy of $H'$ minus an edge.  For Theorem \ref{thm:simulate} we will simulate a $K_6$-process $(\tilde G_i)_{i\geq 0}$ and so  the following fact will be useful. 

    \begin{fact} \label{fact:K6noK7}
    For all $n\in \NN$, there exists a $K_6$-process $(\tilde G_i)_{i\geq 0}$ on an $n$-vertex graph $\tilde G$ such that $\tau_{K_6}(\tilde G)\geq \tfrac{n^2}{2000}$ and for all $i\geq0 $, $\tilde G_i$ is $K_7$-free. 
    \end{fact}

Fact \ref{fact:K6noK7} follows from the work of  Balogh, Kronenberg, Pokrovskiy and the third author of the current paper~\cite{balogh2019maximum}. Indeed one can check that their construction of a $K_6$-process which establishes \cite{balogh2019maximum}*{Lemma 11} is $K_7$-free throughout.
In more detail, the definition of the process $(\tilde G_i)_{i\geq 0}$ in \cite{balogh2019maximum} is defined  via a hypergraph $\cH$ and it can be checked \cite{balogh2019maximum}*{Lemma 8} that any edge that appears in the process $(\tilde G_i)_{i\geq 0}$ is an edge of the \emph{2-skeleton} (see \cite{balogh2019maximum}*{Definition 6}) of the hypergraph $\cH$. It is then shown in \cite{balogh2019maximum}*{Claim 14} that there is a partition $V(\tilde G)=V(\cH)=U\cup U'$ of the vertex set of $\cH$ such that any clique in the $2$-skeleton of $\cH$ has at most 3 vertices in $U$ and at most 3 vertices in $U'$. This therefore implies that the clique number of $\tilde G_i$ is at most 6 throughout the process. 

    \begin{figure}[ht] \label{fig:H1}
    \centering
    \includegraphics[scale=1]{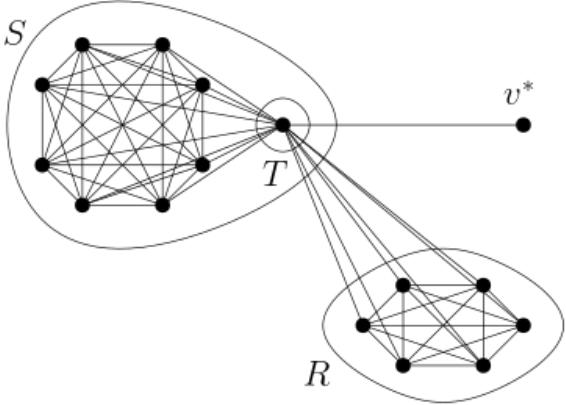}
    \caption{   \label{fig:sim} The graph $H_1$.} 
    \end{figure}

In the proof of Theorem \ref{thm:simulate}, we will use the $K_6$-process $(\tilde G_i)_{i\geq 0}$ given by Fact \ref{fact:K6noK7}. For $1\leq t\in \NN$, we now define the graph $H_t$ with $\delta(H_t)=t$, which we will use to create a $H_t$-process that simulates the $K_6$-process $(\tilde G_i)_{i\geq 0}$. An example when $t=1$ is given in Figure \ref{fig:sim}.  The vertex set is defined by the partition  \[V(H_t):=S\cup R\cup \{v^*\}\] with $|S|=t+8$ and $|R|=6$, and so $|V(H_t)|=t+15$. We also define a vertex subset $T\subseteq S$ with $|T|=t$. The edge set of $H_t$ is then defined as 
\[E(H_t):=\binom{S}{2}\cup \binom{R}{2}\cup \{xy:x\in T, y\in R\} \cup \{xv^*:x\in T\}.\]
In particular note that the vertices $R$ host a $K_6$ and that $\delta(H_t)=\deg _{H_t}(v^*)=t$. 

We are now in a position to prove Theorem \ref{thm:simulate}. 

    \begin{proof}[Proof of Theorem \ref{thm:simulate}]
    For $1\leq t\in \mathbb{N}$, fix $H=H_t$ as defined above and so $\delta(H)=t$ as required. For $n\in \NN$ sufficiently large, fix $\tilde n=n-t-8$ and let $\tilde G$ be the $\tilde n$-vertex graph given by Fact \ref{fact:K6noK7}. We will define a $H$-process $(G_i)_{i\geq 0}$ on a graph $G$ that simulates the $K_6$-process $(\tilde G_i)_{i\geq 0}$ on $\tilde G$. This will complete the proof as by Observation \ref{obs:simulate}, we will have that \[M_H(n)\geq \tau_H(G)=\tau_{K_6}(\tilde G)\geq \frac{\tilde n^2}{2000}=\Omega(n^2),\]
    as required. 

    It therefore remains to define an $n$-vertex starting graph $G$ and prove that the $H$-process on $G$ simulates $(\tilde G_i)_{i\geq 0}$. The graph $G$ is defined by $V(G):=V(\tilde G)\cup S'$ where $|S'|$ is a set of $n-\tilde n=t+8$ vertices disjoint from $V(\tilde G)$. We also identify a subset $T'\subseteq S'$ of size exactly $t$ and the edges of $G$ are defined to be the edges of $\tilde G$, all the edges between vertices of $S'$ and all the edges between $V(\tilde G)$ and $T'$. Thus    $G$ induces a clique on $S'$ and between $T'$ and $V(\tilde G)$ there is a complete bipartite graph. The only other edges are those given by the induced copy of $\tilde G$ and in particular $G$ satisfies \eqref{item:sim 1} of Definition \ref{def:simulation}. 

    To prove property \eqref{item:sim 2} of Definition \ref{def:simulation} note first that any copy of $K_6^-$ (that is the graph obtained by removing an edge from $K_6$) that is induced on the vertices of $\tilde G$ in the $H$-process $(G_i)_{i\geq 0}$, can be extended to a copy of $H^-$ using only edges of $G$. Indeed, we can use $S'$ to play the role of $S$ in $H$, $T'$ to play the role of $T$ and any vertex of $\tilde G$ that does not lie in the copy of $K_6^-$, to play the role of $v^*$. A simple induction  therefore implies that all the edges occurring in the $K_6$-process $(\tilde G_i)_{i\geq 0}$ will also appear in $(G_i)_{i\geq 0}$, that is, for all $i\geq 0$, $E(\tilde G_i)\setminus E(\tilde G)\subseteq E(G_i)\setminus E(G)$. 

    To prove that $E(G_i)\setminus E(G)\subseteq E(\tilde G_i)\setminus E(\tilde G)$ for all $i\geq 0$, we proceed again by induction on $i\geq 0$, noting that the base case $i=0$ trivially holds. So suppose that the claim holds up to time $i-1$ and consider an edge $e\in E(G_i)\setminus E(G_{i-1})$. There is therefore some edge $e'$ in $E(H)$ and an embedding $\phi:H-e'\rightarrow G_{i-1}$ such that $\phi$ maps the vertices of $e'$ to $e$. We claim that $\phi$ must map $S$ to $S'$ and $T$ to $T'$. Indeed $H-e'$ must contain a copy of  $K^-_{8+t}$ and in particular a copy of $K_{7+t}$ on the vertices of $S$. By induction the only edges in $G_{i-1}$ that have been added are those corresponding to the $K_6$-process $(\tilde G_i)_{i\geq 0}$ up to time $i-1$ and so by Fact \ref{fact:K6noK7}, the largest clique in $G_{i-1}$ on the vertices of $\tilde G$ has size $6$. This in turn implies that the largest clique in $G_{i-1}$ that intersects $V(\tilde G)$ has size at most $6+t$. Therefore the image  under $\phi$ of the copy of $K_{7+t}$ in $H-e'$ must  lie entirely in $S'=V(G)\setminus V(\tilde G)$. 
    By considering the $G_{i-1}$-neighbourhoods of the vertices in $S'$ (which are the same as their $G$-neighbourhoods), we quickly see that $\phi$ must map all of  $S$ to $S'$ and $T$ to $T'$ as claimed. Moreover, as all the edges between vertices of $S'$ are present in $G_{i-1}$ and the vertices of $T'$ are universal in $G_{i-1}$, we have that the missing edge $e'$ must be an edge that lies on the vertices $R$ of $H$ and $\phi|_R$ is an embedding of $K_6^-$ into $G_{i-1}[V(\tilde G)]=\tilde G_{i-1}$ which is completed to a copy of $K_6$ by the new edge $e$.  Hence $e\in E(\tilde G_i)\setminus E(\tilde G)$ as required. This completes the proof showing that $(G_i)_{i\geq 0}$ indeed simulates $(\tilde G_i)_{i\geq 0}$. 
    \end{proof}

In the proof of Theorem \ref{thm:simulate}, we used a large clique as a subgraph of $H$ to be `locked in' throughout the $H$-process $(G_i)_{i\geq 0}$. This worked because the process $(\tilde{G}_{i})_{i\geq 0}$ that we simulated was free of large cliques throughout. In the context of Theorem \ref{thm:counterexample} we cannot afford to use a large clique as  our subgraph $H$ has to be sparse. The following lemma gives a process $(\tilde G_i)_{i\geq 0}$ which we can simulate which is $K_3$-free throughout. In fact even more than this, the graphs $\tilde G_i$ are all bipartite. 

    \begin{lem} \label{lem:cyclechord}
    Define $\tilde H$ to be the $6$-vertex cycle with a chord added between two fixed opposite vertices. That is,  $V(\tilde H) = [0,5]$ and $E(\tilde H) = \{ xy : x-y \equiv 1 \mod 6 \} \cup \{ \{0,3\} \}$. Then for every $n\in \NN$,  there exists a $\tilde H$-process $(\tilde G_i)_{i\geq 0}$ on an $n$-vertex graph $\tilde G$ such that $\tau_{\tilde H}(\tilde G)=\Omega(n)$ and $\tilde G_i$ is bipartite for all $i\geq 0$. 
    \end{lem}
    \begin{proof}
    By the `moreover' part of Theorem \ref{thm:23bound}, as $\tilde H$ is bipartite and has minumum degree $\delta(\tilde H)=2$ and maximum degree $\Delta(\tilde H)=3$, we have that there is some $n$-vertex bipartite $\tilde G$ such that the $\tilde H$-process $(\tilde G_i)_{i\geq 0}$ on $\tilde G$ lasts linearly many steps. 

    It remains to establish that each $\tilde G_i$ is bipartite. Suppose for a contradiction that this is not the case and let $i^*$ be the smallest $i\geq 1$ such that $\tilde G_i$ is non-bipartite. Hence there is a new edge $e$ added within one of the parts of $\tilde G$ at step $i^*$ and the edge $e$ completes a copy of $\tilde H$. Every edge of $\tilde H$ is contained in a cycle in $\tilde H$ and so the edge $e$ must form the endpoints of a path in $\tilde G_{i^*-1}$. As the vertices of $e$ lie in the same part of a bipartition of $\tilde G_{i^*-1}$, this path must have even length. This  path along with $e$ give an odd cycle in the embedding of $\tilde H$, contradicting that $\tilde H$ is itself bipartite. 
    \end{proof}

In the following we refer to the graph obtained by deleting an arbitrary edge from $K_4$ as the \emph{diamond}. The role of the subgraph $H'\subseteq H$ that will be locked in place throughout our $H$-process $(G_i)_{i\geq 0}$ will be played by the following graph, a visualisation of which  is shown in Figure \ref{fig:counter_hprime}. 

    \begin{figure}[ht]
    \centering
    \includegraphics[width=0.7\linewidth]{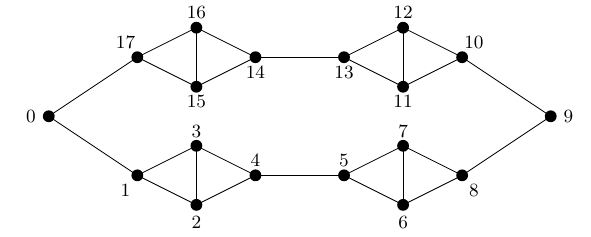}
    \caption{A drawing of $H'$. The vertex $u_i$ is labelled by $i$.}
    \label{fig:counter_hprime}
    \end{figure}

Let $H'$ be a graph with vertices $u_i$, $i\in [0,17]$, and
    \[
    E(H') := \left\lbrace u_iu_{i+1} : i \in [0,16] \right\rbrace \cup \left\lbrace u_{17}u_0,u_1u_3, u_2u_4, u_5u_7, u_6u_8, u_{10}u_{12}, u_{11}u_{13}, u_{14}u_{16}, u_{15}u_{17} \right\rbrace.
    \]
Let 
    \begin{equation}\label{eq:counter_partition}
    U := \{ u_0,\ldots,u_9 \} \quad\text{ and }\quad W:=\{ u_9,\ldots,u_{17},u_0 \}, 
    \end{equation}
so $H'[U]$ and $H'[W]$ are isomorphic and contain two vertex-disjoint diamonds each.

    \begin{obs} \label{obs:H'-self-stable }
    The graph $H'$ is itself $H'$-stable.
    \end{obs}
    \begin{proof} 
    This is a simple case of checking that for every $e_1\in E(H')$ and $e_2\in \binom{V(H')}2 \setminus E(H')$, $H'$ and $H' - e_1 + e_2 $ are not isomorphic.
    We leave the details for the reader.
    \end{proof}

We now define $H$ which contains both $H'$ (with vertices labelled as above)  and $\tilde H$ as induced subgraphs, has an edge between $u_9$ and a degree $2$ vertex of $\tilde H$ and an edge between $u_0$ and a new vertex which we label $z\in V(H)$. See Figure \ref{fig:fullcounterexample} for a drawing of $H$. Clearly $\delta(H)=1$ and $\Delta(H)=3$ as required. 

    \begin{figure}[ht]
    \centering
    \includegraphics[width=0.7\linewidth]{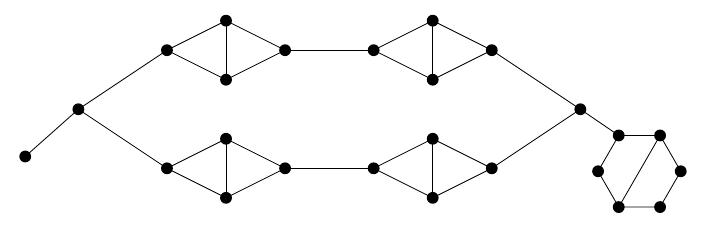}
    \caption{A drawing of $H$. }
    \label{fig:fullcounterexample}
    \end{figure}

We are now ready to prove Theorem \ref{thm:counterexample}.

    \begin{proof}[Proof of Theorem \ref{thm:counterexample}]
    Take $n\in \NN$ with $n\geq 19$ and let $n'=n-18$.  Fix $\tilde H$ as in Lemma \ref{lem:cyclechord} and let $\tilde G$ be a starting graph with $n'$ vertices such that the $\tilde H$-process $(\tilde G_i)_{i\geq 0}$ on  $\tilde G$ has $\tau_{\tilde H}(\tilde G)=\Omega(n')=\Omega(n)$ and $\tilde G_i$ is bipartite for all $i\geq 0$, the existence of which is given by Lemma \ref{lem:cyclechord}.

    We now construct a graph $G$ on $n$ vertices.
    Let $H'$ be as defined above, take $H'_0$ to be isomorphic to $H'$ and vertex-disjoint from $\tilde G$, and fix an isomorphism $\varphi': H' \to H'_0$.
    Set $u'_0 := \varphi'(u_0)$, $u'_9 :=\varphi'(u_9)$ and define $G$ via
        \begin{linenomath}\begin{equation}\begin{split}\label{eq:counter_defofG}
        V(G) &= V(\tilde G) \cup V(H'_0),  \\
        E(G) &= E(\tilde G) \cup E(H'_0) \cup \left\lbrace u'_0x : x\in V(\tilde G) \right\rbrace \cup \left\lbrace u'_9x : x\in V(\tilde G) \right\rbrace.
        \end{split}\end{equation}\end{linenomath}

    Finally we let $H$ be as defined above (Figure \ref{fig:fullcounterexample}) and as in the proof of Theorem \ref{thm:simulate}, we will show that the $H$-process $(G_i)_{i\geq 0}$ on $G$ simulates the $\tilde H$-process $(\tilde G_i)_{i\geq 0}$ on $\tilde G$. This then suffices to prove Theorem \ref{thm:counterexample} due to Observation \ref{obs:simulate}. 

    Note that condition \eqref{item:sim 1} of Definition \ref{def:simulation} is certainly  fulfilled by construction and, as with Theorem \ref{thm:simulate}, the inclusion $E(\tilde G_i)\setminus E(\tilde G)\subseteq  E(G_i)\setminus E(G)$   for all $i\geq 0$ follows from induction and  the fact that for any $e\in E(\tilde H)$, any copy $K$ of  $\tilde H-e$ on the vertices $V(\tilde G)$ can be extended to a copy of $H-e$ using $\phi'(H')\subseteq G$ and an arbitrary vertex $v\in V(\tilde G)\setminus V(K)$ to play the role of $z\in V(H)$. 

    Finally then, it remains to show that 
        \begin{linenomath}\begin{equation}\label{eq:counter_stronger}
        E(G_i)\setminus E(G) \subseteq E(\tilde G_i)\setminus E(\tilde G) \qquad\text{for each } i\geq 0.
        \end{equation}\end{linenomath} 
    We proceed by induction on $i\geq 0$, noting that the base case $i=0$ is vacuously true.
    So suppose that $i\geq 1$ and $E(G_{i-1})\setminus E(G) \subseteq E(\tilde G_{i-1})\setminus E(\tilde G)$.
    Any copy of $H-e$, where $e\in E(H)$, contains a copy of $H'$ or $H'-e$.
    For this reason the following claim will be helpful.

        \begin{clm}\label{clm:counter_embedding}
        Let $e'\in E(H')$.
        Every embedding $\phi:H'-e'\to G_{i-1}$ satisfies $\phi(V(H')) = V(H'_0)$ and $\{\phi(u_0), \phi(u_9)\} = \{u'_0, u'_9\}$.
        \end{clm}
        \begin{proof}
        Let $\phi:H'-e'\to G_{i-1}$ be an embedding, and let $H'' := \phi(H'-e')$. Now the only diamonds in $G_{i-1}$ are the four vertex disjoint ones in $H'_0$ which we call \emph{type 1 diamonds}, those which use either $u_0'$ or $u_9'$ as a degree 3 vertex and have all other vertices in $V(\tilde G)$, which we call \emph{type 2 diamonds} and finally, the \emph{type 3 diamonds} which are those which use \emph{both} $u_0'$ and $u_9'$ as degree 2 vertices and have degree $3$ vertices in $V(\tilde G)$. Indeed, the fact that there are no diamonds in $G_{i-1}$ using just one of $u_0'$ and $u'_9$ as a degree 2 vertex follows from the fact that $G_{i-1}[V(\tilde G)]=\tilde G_{i-1}$   by induction and so $G_{i-1}[V(\tilde G)]$ is bipartite. 
   
        Now note  that there are disjoint diamonds $K_1,K_2, L$ in  $H''\subseteq G_{i-1}$ and some vertex $w\in \{\phi(u_0), \phi(u_9)\}$ that is adjacent in $H''\subseteq G_{i-1}$ to degree 2 vertices in both $K_1$ and $K_2$. We claim that $w\in \{u_0',u_9'\}$. Indeed, if this was not the case then the only possibility for $G_{i-1}$ to have $w$ incident to degree two vertices of disjoint diamonds $K_1$ and $K_2$ is if $K_1$ and $K_2$ are both type 2 diamonds and $w\in V(\tilde G)$. In that case, the third disjoint diamond $L$ must be a type 1 diamond.  However then $L$ is disconnected from $w$ in $H''$ as both the vertices $u_0'$ and $u'_9$ already have degree 3 in the diamonds $K_1$ and $K_2$ and so none of the edges incident to $u_0'$ or $u_9'$ in $H_0'$ can be present in $H''$. This contradicts the fact that $H''$ is connected (which follows as $H'$ is 2-connected). 

        So we have that $w\in \{u_0',u_9'\}$ and we now  let $w'\in \{\phi(u_0),\phi(u_9)\}$ be such that $w'\neq w$. Consider the graph $F$ which is isomorphic to $H'[U]$ and $H'[W]$ (see \eqref{eq:counter_partition}) and note that there must be a copy of $F$ in $H''\subseteq G_{i-1}$ which has $w$ and $w'$ as its degree 1 vertices. As $w\in \{u'_0,u_9'\}$, the only way to embed $F$ with $w$ as a degree one vertex is to use two of the type 1 diamonds for $F$ and have $w'\in \{u'_0,u_9'\}\setminus \{w\}$. Indeed, if the diamond closest to $w$ in the copy of $F$ is type 2, then there is no way to find the second diamond in $F$ which is adjacent to the first and disjoint from $w$  and the first diamond. So both diamonds in the copy of $F$ are type 1 and $\{w,w'\}=\{\phi(u_0),\phi(u_9)\}=\{u'_0,u'_9\}$ as required. Finally, by considering the fact that all triangles in $G_{i-1}$ are inside type 1 diamonds or contain one of $u'_0$ and $u'_9$, we see that $H''=\phi(H'-e')$ must indeed be contained in $V(H_0')$, completing the proof of the claim.
        \end{proof}

    Let $\tilde e \in E(G_i)\setminus E(G)$.
    Our goal is to show that $\tilde e \in E(\tilde G_i)\setminus E(\tilde G)$.
    We can assume $\tilde e \in E(G_i)\setminus E(G_{i-1})$ as otherwise we would be done by the induction hypothesis.
    Therefore there exists an edge $e\in E(H)$ and an embedding $\varphi: H \to G_i$ such that $\varphi(e) = \tilde e$ and $\varphi(H-e) \subseteq G_{i-1}$.
    Let $e' := e$ if $e \in E(H')$ and an arbitrary edge of $H'$ otherwise.
    Since $\varphi_{|V(H')}$ can be regarded as an embedding of $H'-e'$ into $G_{i-1}$ we may apply Claim \ref{clm:counter_embedding} to obtain $\varphi(V(H')) = V(H'_0)$ as well as $\{ \varphi(u_0),\varphi(u_9) \} = \{ u'_0, u'_9 \}$.
    Since $H'$ is $H'$-stable by Observation \ref{obs:H'-self-stable }, we conclude that $e\notin E(H')$.
    The only edges of $H$ with precisely one endpoint in $V(H')$ are $u_0z$ and $u_9\tilde v$ where $\tilde v$ is the unique vertex of $V(\tilde H)$ connected to $V(H')$ in $H$.
    However, both $u'_0$ and $u'_9$ are already adjacent to every vertex in $V(\tilde G)$ at time $0$, so the endpoints of $e$ must lie in $V(\tilde H)$.
    Then $\varphi_{|V(\tilde H)}$ is an embedding of $\tilde H$ in $\tilde G_i$ with $\varphi_{|V(\tilde H)}(\tilde H-e) \subseteq \tilde G_{i-1}=G_{i-1}[V(\tilde G)]$.
    This implies $\tilde e \in E(\tilde G_i)$, which completes the induction and thus shows that $(G_i)_{i\geq 0}$ does indeed simulate the process $(\tilde G_i)_{i\geq 0}$, completing the proof of the theorem.
    \end{proof}

\section{A linear upper bound under restricted connectivity} \label{sec:girth}
In this section we give the proof of Theorem \ref{thm:girth_construction} giving a linear upper bound on $M_H(n)$ for graphs $H$ that have restricted connectivity and are self-percolating, that is, the $H$-process on $H$ itself stabilises at a clique $K_{v(H)}$. We remark that our proof develops the method of Bollob\'as, Przykucki, Riordan and Sahasrabudhe~\cite{bollobas2017maximum} in showing that $M_{K_4}(n)\leq n-3$, which shows that if the process does not stabilise, then there is a clique that grows throughout the process. In their proof, they show that each step leads to a larger clique whilst we will only be able to show that the intervals between times at which the clique grows are bounded by some constant.    Before embarking on the proof, we give a simple lemma that holds for any   $H$-process $(G_{i})_{i\geq 0}$ and identifies a connected sequence of copies of $H$ such that the $i^{th}$ copy is completed at time $i$.

    \begin{lem} \label{lem:seq}
    Let $H$ be a graph and suppose $(G_i)_{i\geq 0}$ is a $H$-process on a starting graph $G$, with  $\tau:=\tau_H(G)$ the time at which the process stabilises. Then there exists a sequence $H_1,\ldots, H_\tau$ of copies of $H$ on $V(G)$ such that   $H_i$ is completed at time $i$ for all $i\in [\tau]$ and $|V(H_i)\cap V(H_{i+1})|\geq 2$ for $1\leq i\leq \tau -1$.
    \end{lem}   
    \begin{proof}
    The lemma is trivially true when $\tau(H)\leq 1$ so we can assume otherwise.
    In particular, $e(H)\geq 2$. 
    We will define the sequence inductively in reverse order.  Let $H_\tau$ be some copy of $H$ that is a subgraph of $G_\tau$ but not $G_{\tau-1}$, noting that such a copy must exist by the definition of $\tau$. Now suppose that $1\leq t < \tau$ and for all $s$ such that $t<s\leq \tau$, we have found an appropriate $H_s$ that is completed at time $s$. In particular, there is an edge $e\in E(H_{t+1})$ that completes $H_{t+1}$ at time $t+1$ in the process $(G_i)_{i\geq 0}$. We claim that there is another edge $e'\in E(H_{t+1})\setminus \{e\}$ that is in $E(G_{t})\setminus E(G_{t-1})$. Indeed, if this was not the case, then all edges of $H_{t+1}$ apart from $e$ would be present at time $t-1$ and thus edge $e$ would be added at time $t$, a contradiction. Finally then, we define $H_t$ to be a copy of $H$ completed by edge $e'$ at time $t$. Defining such $H_t$ for all $1\leq t < \tau$ in reverse order thus defines the sequence as in the statement of the lemma.
    \end{proof}

We also need a simple observation about self-percolating graphs. 

    \begin{obs} \label{obs:self perc}
    Let $H$ be a graph such that $\final{H}_H=K_{v(H)}$. Then for any graph $G$ such that $H\subseteq G$ and $v(G)=v(H)$, we have that $\final{G}_H=K_{v(H)}$.
    \end{obs}
    \begin{proof}
    Let $k=v(H)$. Recall from Observation \ref{obs:final} that for any $G'$ on $k$ vertices, we have that $\final{G'}_H$ is the smallest $k$-vertex $H$-stable graph in which $G'$ appears as a subgraph. Hence, as $\final{H}_H=K_{k}$, we have that there are no incomplete $k$-vertex graphs which contain $H$ as a subgraph and are $H$-stable. As $H\subseteq G$ for $G$ as in the statement, again appealing to Observation \ref{obs:final} gives that indeed  $\final{G}_H=K_{k}$. 
    \end{proof}

We are now ready to give the proof of Theorem \ref{thm:girth_construction}. 

    \begin{proof}[Proof of Theorem \ref{thm:girth_construction}]
    Fix $H$ and $e\in E(H)$ so that $\kappa (H-e)\leq 2$ and $\final{H}_H=K_{v(H)}$, as in the statement of the theorem. Moreover, let $n\in \NN$, let $G$ be an  $n$-vertex graph with $\tau:=\tau_H(G)=M_H(n)$ and let $(G_i)_{i\geq 0}$ be the $H$-process with starting graph $G$. Finally, fix $c=v(H)^2$. 

    The following claim is at the heart of our argument. 
        \begin{clm}\label{clm:superlinear_3connected}
        Suppose $0\leq i\leq \tau-c$ and there are vertex sets  $U,W\subset V(G)$ each of size at least $v(H)$ and such that $|U\cap W|\geq 2$ and both $G_i[U]$ and $G_i[W]$ are cliques.  Then $U\cup W$ hosts a clique at time $i+c$.
	\end{clm}	
	\begin{proof}
        Let $x\in U\setminus W$ and $y\in W\setminus U$ such that $xy\notin E(G_i)$ (if such a pair does not exist, the result trivially holds). We claim that there is a copy $H'$ of $H-e$ in $G_i$ such that $x,y\in V(H')$. Indeed as $\kappa(H-e)\leq 2$, by definition either $H-e$ is a clique of size at most 3 (which cannot occur as we removed an edge from a graph $H$ to get $H-e$) or there is some vertex subset $Z_0\subseteq V(H-e)$ of size at most 2 such that $H-e$ is disconnected after removing $Z_0$. Let $Z_1$ and $Z_2$ be vertex sets of distinct connected components in $H-e$ after removing $Z_0$. Then we can embed $H-e$ into $G_i$ mapping $Z_0$ into $U\cap W$ and such that $x$ lies in the image of $Z_1$ and $y$ lies in the image of $Z_2$. This gives the required copy  $H'$. Now if the pair of vertices in $V(H')\subseteq V(G)$ corresponding to the edge $e$, are not adjacent in $G_i$, they will certainly be adjacent at time $i+1$. Thus $G_{i+1}[V(H')]$ is a $v(H)$-vertex set that contains a copy of $H$ and by Observation \ref{obs:self perc}, the $H$-process on $G_{i+1}[V(H')]$ results in a clique. Moreover, trivially, this happens after at most $c-1$ time steps. Thus by Observation \ref{obs:hom}, $G_{i+c}[V(H')]$ is a complete graph and, in particular, $xy\in E(G_{i+c})$. As $x$ and $y$ were arbitrary nonadjacent vertices at time $i$, this proves the claim. 
        \end{proof}
    Let $H_1,\ldots H_\tau$ be the sequence of copies of $H$ in the process $(G_i)_{i\geq 0}$ given by Lemma \ref{lem:seq}. Using Claim \ref{clm:superlinear_3connected}, we have that  the vertices $V(H_1\cup\ldots\cup H_\tau)$  host a clique in $\final{G}_H$. Indeed, this follows by a simple induction showing that for all $i\in [\tau]$, $\final{G}_H[V(H_1\cup\ldots\cup H_i)]$ is complete. The    base case asserts $V(H_1)$ hosts a complete graph during the process which follows from Observation \ref{obs:self perc} (and Observation \ref{obs:hom}). For $2\leq i\leq \tau$, by induction we can assume that $V(H_1\cup\ldots \cup H_{i-1})$ is complete at some point in the process and as with $H_1$, Observation \ref{obs:self perc} implies that $V(H_i)$ will host a complete graph. Claim \ref{clm:superlinear_3connected} then implies that $\final{G}_H[V(H_1\cup\ldots\cup H_i)]$ is indeed complete as $|V(H_i)\cap (V(H_1\cup\ldots \cup H_{i-1}))|\geq |V(H_i)\cap V(H_{i-1})|\geq 2$. 

    In order to bound the running time $\tau$, we bound the time which it takes for these cliques to form. We identify some key time steps as follows. First we let $i_0=1$ and for $j\geq 1$, we define 
        \[
        i_j:=\min\{i\in[\tau]:V(H_i)\nsubseteq V(H_1\cup\ldots\cup H_{i_{j-1}})\},
        \]
    where we use the convention that $\min \emptyset=\infty$.  Let $j^*$ be the maximum $j\geq 0$ such that $i_j\neq \infty$ and note that $j^*\leq n$ as for $1\leq j\leq j^*$, we have that $|V(H_1\cup\ldots\cup H_{i_{j}})|>|V(H_1\cup\ldots\cup H_{i_{j-1}})|$ and $|V(H_1\cup\ldots\cup H_{i_{j^*}})|\leq |V(G)|=n$. For $0\leq j\leq j^*$, we now also define 
        \[
        t_j:=\min\{i\in [\tau]:G_i[V(H_1\cup\ldots\cup H_{i_{j}})] \mbox{ is complete}\},
        \]
    noting that this is well-defined as we showed above that $\final{G}_H[V(H_1\cup\ldots\cup H_\tau)]$ is complete. 

    We claim that $t_{j^*}=\tau$. Indeed, the fact that $t_{j^*}\leq\tau$ follows from the fact that an edge was added at time $t_{j^*}$ to complete the clique on $V(H_1\cup\ldots\cup H_{i_{j^*}})$. Moreover, if $\tau>t_{j^*}$, then the edge added to complete $H_\tau$ would be added after $V(H_1\cup\ldots\cup H_{i_{j^*}})$ already hosts a clique and thus $V(H_\tau)$ cannot be contained in $V(H_1\cup\ldots\cup H_{i_{j^*}})$, contradicting the definition of $j^*$. 

    We now provide upper bounds on the indices $t_j$. 

        \begin{clm} \label{clm:index upper bd}
        For $1\leq j \leq j^*$, we have that 
        $t_j\leq t_{j-1}+2c. $
        \end{clm}
        \begin{proof}
        We first note that for all $j\in[j^*]$, we have 
            \begin{equation}\label{eq:ij upper}    
            i_j\leq t_{j-1}+1.
            \end{equation}
        Indeed, at time $t_{j-1}$ the set $V(H_1\cup\ldots\cup H_{i_{j-1}})$ hosts a clique and so the edge added to complete $H_{t_{j-1}+1}$ must contain at least one vertex outside $V(H_1\cup\ldots\cup H_{i_{j-1}})$, and thus the first $i$ such that $V(H_i)$ is not contained in $V(H_1\cup\ldots\cup H_{i_{j-1}})$, is at most $t_{j-1}+1$.

        We will now show that for all $j\in [j^*]$, 
            \begin{equation} \label{eq:tj upper bd}
            t_j\leq \max\{i_j+c-1,t_{j-1}\}+c,
            \end{equation}
        from which the claim follows immediately using \eqref{eq:ij upper}. To prove, \eqref{eq:tj upper bd}, note that at time $i_j+c-1$, we will have that $U':=V(H_{i_j})$ hosts a clique due to Observation \ref{obs:self perc} (and Observation \ref{obs:hom}). Moreover, by definition at time $t_{j-1}$ the set $W':=V(H_{1}\cup\ldots\cup H_{i_{j-1}})$ hosts a clique. Hence at time $t':=\max\{i_j+c-1,t_{j-1}\}$, $U'$ and $W'$ both host cliques and \eqref{eq:tj upper bd} follows from an application of Claim \ref{clm:superlinear_3connected} using that $|U'\cap W'|\geq |V(H_{i_j})\cap V(H_{i_{j}-1})|\geq 2$. 
        \end{proof}
        
    Finally, Claim \ref{clm:index upper bd} implies that 
        \[
        \tau=t_{j^*}\leq 1+2cj^*\leq 1+2cn=O(n),
        \]
    as required. 
    \end{proof}

\section{Concluding remarks} \label{sec:conclude}
To conclude, we give some remarks on extensions of our results and directions for future research. 

\subsection{Trees}
For all $t$-vertex trees $T$, Theorem \ref{thm:trees} gives an upper bound on $M_T(n)$ that is quadratic in $t$ and independent of $n$.  A natural question is to consider how tight this result is.

    \begin{quest} \label{q:tree}
    For $t\in \NN$, what is $M^*(t):=\max\{\limsup_{n\rightarrow \infty}M_T(n): T \mbox{ is a } t\mbox{-vertex tree}\}$?
    \end{quest}

We focus on the $\limsup$ here as $M_T(n)$ is not necessarily monotone increasing in $n$ and there may be small values of $n$  which we are not interested in\footnote{We are also unsure whether the sequence might oscillate.}. With this notation, Theorem \ref{thm:trees} gives that $M^*(t)=O(t^2)$. As a lower bound the following simple example gives that $M^*(t)\geq t-1$. 

    \begin{exmp} \label{ex:star}
    Let $T=K_{1,t-1}$ be the $t$-vertex star. Then $M_T(n)=t-1$ for all $n$ sufficiently large. 
    \end{exmp}
    \begin{proof}
    The upper bound was already shown in Lemma \ref{lem:star}, so it suffices to prove a lower bound. For this, we take $G$ to be the disjoint union of the stars $K_{1,s}$, $1\leq s \leq t-2$ and a set $W$ of $n-\binom{t}2+1$ isolated vertices. The $K_{1,t-1}$-process $(G_i)_{i\geq 0}$ on $G$ runs for $t-1$ steps before stabilising at the complete graph. Indeed, for $1\leq i \leq t-3$,  exactly one vertex becomes universal at time step $i$ and the degree of all vertices in $W$ increases by 1. At time step $t-2$ all degree-one vertices in $G$ become universal and $W$ remains an independent set. Finally at time step $t-1$, all missing edges are added.  \end{proof}

 We actually expect that this example is best possible and pose the following conjecture, which would establish that $M^*(t)=t-1$. 

    \begin{conj} \label{conj:tree}
 For all $t\in \NN$ and $n\in \NN$ sufficiently large,   $M_T(n)\leq t-1$ for any $t$-vertex tree $T$. 
    \end{conj}

    Evidence for this conjecture is given by our proof of  Theorem \ref{thm:trees}. Indeed, an inspection of the proof  gives that the quadratic term in our upper bound is only needed to upper bound the parameter $\mu=\mu(T;z)$. Recall from \eqref{eq:trees_mu}  that $\mu$ is  the maximum running time of the $F$-process on $|V(F)|$ vertices, where $F=T[U]$ is the subforest of $T$ induced on the vertex set $U$ given by Lemma \ref{lem:trees_vertexcover}.  We lack any effective bound on this parameter $\mu$ and so simply upper bound $\mu$ by the total number of possible edges on $|U|\leq t/2$ vertices. Our proof thus reduces our current upper bound on $M^*(t)$ to a question of bounding the running time of a forest $F$ on $v(F)=f$ vertices, and any progress on the trivial bound of $\binom{f}{2}$ would lead to an improvement on the upper bound of $M^*(t)$. The following question asks this explicitly and we remark that a natural subcase  to consider first would be when the forests are themselves trees.

    \begin{quest} \label{q:forest}
    For $f\in \NN$, what is  the maximum value of $M_F(f)$ over all  $f$-vertex forests $F$? 
    \end{quest}

 A linear upper bound for Question \ref{q:forest} would result in a linear upper bound for Question \ref{q:tree}, thus matching the order of magnitude of our lower bound from Example \ref{ex:star}.

 \begin{rem} \label{rem:forest}
     In a previous version of this manuscript, we asked if one could get a bound of the form $M_F(f)\leq cf$ for some $c<1/2$ for Question \ref{q:forest} which would directly give Conjecture \ref{conj:tree} due to our upper bound in Theorem \ref{thm:trees} which has a linear term of the form $3t/4$ (when excluding the case that $T$ is a star). The anonymous referee brought our attention to a construction giving that $M_{K_{1,f-1}}(f)\geq f-3$. Namely, one takes as $G_0$ the graph on vertex set $[f]$ such that the only non-edges are $\{\{i,i+1\}:i=1,\ldots,f-1\}$ and $\{f-2,f\}$. Each vertex $i=1,\ldots,f-3$ will then become universal in step $i$ of the process and so the running time is $f-3$. 

As every forest $F$ can be realised as $T[U]$ for some tree $T$
 and vertex set $U$ as given by Lemma \ref{lem:trees_vertexcover}, this shows that one cannot establish Conjecture \ref{conj:tree}, simply by universally bounding the quantity $\mu=\mu(T;z)=M_{T[U]}(|U|)$ from \eqref{eq:trees_mu}. 
 Nonetheless, we believe that there is a route towards proving Conjecture \ref{conj:tree} by considering each possible tree and individually bounding the  parameters appearing in  \eqref{eq:trees_finalbutone} which correspond to the different stages of the process in our analysis, with specialised bounds depending on the tree in question.   
  \end{rem}

\subsection{The sublinear regime}
Our work here makes significant progress on understanding the graphs $H$ for which $M_H(n)$ can be sublinear. Indeed, by Theorem \ref{thm:23bound}, any connected\footnote{The connected condition here is crucial, see Section \ref{sec:disconnect}.}  graph  $H$ with a running time  $M_H(n)$ that is sublinear in $n$ is either a cycle or has a degree 1 vertex. However, we fall short of characterising those graphs $H$ with sublinear running time and Theorem \ref{thm:simulate} shows that not all graphs $H$ with a degree 1 vertex have this property. For a better understanding of the class of graphs $H$ with sublinear running time, we need a better grasp on the influence of degree 1 vertices on running times. One potential route to do this would be to compare the running time of a graph $H$ with a degree one vertex $v$ and the running time for $H-v$.  In all cases we know, we  either have that  removing the degree one vertex   has no significant effect on the running time (as with the construction for Theorem \ref{thm:simulate}) or  the running time for $H-v$ is large and adding the pendent edge  drops the running time to a constant (as with Example \ref{ex:clique pendent edge}). Moreover, the behaviour is also sensitive to which vertex the pendent edge is added, as our construction for Theorem \ref{thm:simulate} can be seen to have constant running time if the pendent edge is moved to any other vertex of $S$ (see Figure \ref{fig:sim}). It would be very interesting to establish if this is a general phenomenon or if there is some  example where we see  different behaviour.

    \begin{quest} \label{q:deg}
    Is there a graph $H$ and $v \in V(H)$ with $\deg_H(v) = 1$ such that $M_{H}(n) = o(M_{H-v}(n))$ and $M_{H}(n) = \omega(1)$?
    \end{quest}

A negative answer to Question \ref{q:deg} would suggest that any graph $H$ with a degree 1 vertex $v$ either has $M_H(n)$ bounded by a constant or has $M_H(n)$ of the same order as some $M_{H'}(n)$ where $H'\subseteq H$ is a connected subgraph of $H$ with no degree 1 vertices.  This would imply that  apart from constant and logarithmic there are no other types of running time below linear.  We pose this as a conjecture in its own right.

    \begin{conj}\label{conj:ranges}
    Every  $H$ with $M_H(n) = o(n)$ satisfies either $M_H(n) = O(1)$ or $M_H(n) = \Theta(\log n)$.
    \end{conj}

\subsection{Tree-width}\label{sec:tree-width}
In this paper, we initiated the study of which graph parameters of $H$ control the maximum running time of $H$-processes. 
Given Theorem \ref{thm:trees}, another natural parameter to consider is \emph{tree-width}, for which we use the notation $\tw(H)$.  Example \ref{ex:clique pendent edge} shows that there are graphs with large tree-width and constant maximum running time. Indeed, $\tw(H_k)=\tw(K_k)=k-1$. This shows that unlike small minimum degree and small connectivity, small tree-width is \emph{not} necessary for even constant maximum running time. However,  it might still be the case that it is sufficient, that is, a small tree-width implies an upper bound on running time.  Indeed trees are the unique graphs $H$ with $\tw(H)=1$ and thus Theorem \ref{thm:trees} gives that any graph with tree-width 1 has constant running time.  On the other hand, graphs with tree-width 2 can already have linear running time. Indeed, this is evidenced by the fact that  $\tw(K_{k,\ell})=\min\{k,\ell\}$ and Proposition \ref{prop:k2s_upper}, which states that $M_{K_{2,s}}(n)=\Theta(n)$ for $s\geq 3$. We believe that this is the largest running time that a tree-width 2 graph can have and pose the following conjecture. 

    \begin{conj} \label{conj:tw}
    Any graph $H$ with $\tw(H)=2$ has $M_H(n)=O(n)$.
    \end{conj}

It is well-known that  if a graph   $H$ has $\kappa(H)\geq 3$, then $H$ contains a $K_4$-minor. On the other hand, graphs with tree-width at most 2 are precisely the graphs with no $K_4$-minor. Therefore a solution to  Conjecture  
\ref{conj:tw} would give a natural subclass  of the set of  graphs $H$ with $\kappa(H)\leq 2$, that have maximum running time at most linear. This would complement Theorem \ref{thm:girth_construction} which also gives a subclass of graphs with $\kappa(H)\leq 2$ that have running time at most linear.  

Finally we remark that using bounds on tree-width to give effective upper bounds on running time only has the hope to work for very small values of tree-width, namely tree-width 1, as in Theorem \ref{thm:trees} and tree-width 2 as in Conjecture \ref{conj:tw}. Indeed, in forthcoming work \cite{FMSz3}, we give a construction showing that the wheel graph $W_7$ with $7$ spokes,  which has tree-width 3,  already has running time  $M_{W_7}(n)=n^{2-o(1)}$.

\subsection{Disconnected graphs}\label{sec:disconnect}
Most of our results transfer to the case of disconnected $H$ with only some additional technical difficulties. Indeed, Theorem \ref{thm:trees} can be adapted to give exactly the same upper bound for $t$-vertex forests and in \cite{FMSz1} we showed that if $H$ is a union of disjoint cycles, then again $M_H(n)$ is logarithmic in $n$.  

On the other hand, not all our results transfer directly to disconnected $H$, as the following example shows. 

    \begin{exmp} \label{ex:triangles}
    Let $H^\Delta$ be the graph in Figure \ref{fig:23counterexample}. Then $M_{H^\Delta}(n) \leq 4$.
    \end{exmp}
    \begin{proof}
    Given a copy $H_1$ of $H^\Delta$ at time $1$ of the $H^\Delta$-process $(G_{i})_{i\geq 0}$ on some graph $G$, $V(H_1)$ would be a clique at time $2$. Any two edges that are incident at time $i\geq 2$ form a triangle after one more step, and any two triangles at time $i$ form a $K_6$ at time $i+1$. This shows that $G_4$ is a disjoint union of a (possibly empty) collection of isolated vertices, a (possibly empty) collection of isolated edges and one large clique on all the remaining vertices outside these two collections. The graph $G_4$ is thus $H^\Delta$-stable. 
    \end{proof}

Thus it is no longer true for disconnected $H$, that $\delta(H)\geq 2$ and $\Delta(H)\geq 3$ implies a running time that is at least linear, as was the case for connected $H$ by Theorem \ref{thm:23bound}. One can however adapt the proof of Theorem \ref{thm:23bound} to show that if $H$ has the property that \emph{every connected component} of $H$ has minimum degree at least 2 and maximum degree at least 3, then again $M_H(n)=\Omega(n)$. 

    \begin{figure}
    \centering
    \includegraphics[width=0.5\linewidth]{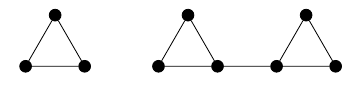}
    \caption{A graph with maximum running time $O(1)$ whose two components have logarithmic or linear running time, respectively.}
    \label{fig:23counterexample}
    \end{figure}

In general, when  $H$ is the disjoint union of two graphs $H_1$ and $H_2$ we do not know to what extent $M_H(n)$ depends on the individual running times $M_{H_1}(n)$ and $M_{H_2}(n)$.
We have encountered examples for which the asymptotic running time of $H$ matches one of the individual running times as well as examples where $M_{H_1}(n)$ and $M_{H_2}(n)$ are  large while $M_H(n)$ is very small (as in Example \ref{ex:triangles}).
However, we have not seen whether $M_H(n)$ can be much bigger than both $M_{H_1}(n)$ and $M_{H_2}(n)$.

    \begin{quest}
    Are there graphs $H_1$, $H_2$ such that their disjoint union $H_1\sqcup H_2$ satisfies
    $M_{H_1\sqcup H_2}(n) = \omega\left(M_{H_1}(n)+M_{H_2}(n)\right)$?
    \end{quest}

\subsection{Small graphs} We conclude by remarking that our results here   give a fairly complete picture of the running times for small graphs $H$. Indeed for example, using our results as well as some more ad-hoc proofs, we can show that apart from $K_5$ all graphs on at most 5 vertices have running time at most linear in $n$. Determining the running time of $K_5$, and in particular whether $M_{K_5}(n)$ is quadratic or not, remains one of the most pertinent open questions in this area.

\subsection*{Acknowledgements} The authors would like to thank the anonymous referees for their careful reading of the previous manuscript and their suggestions. In particular, we are grateful for the example provided by one of the referees which we discuss in Remark \ref{rem:forest}.

\bibliography{Biblio}
\end{document}